\documentclass[11pt,reqno]{amsart}
\usepackage{amsfonts}
\usepackage{amsmath,amsthm,amssymb,url,listings,amsfonts,amscd}
\usepackage[alphabetic,initials]{amsrefs}
\usepackage{mathrsfs}
\usepackage{lipsum}
\usepackage{bbding}
\usepackage{graphicx,latexsym}
\usepackage{hyperref}
\usepackage{scrtime}
\usepackage{color}

\newcommand{\supp}{\operatorname{supp}}
\newcommand{\bea}{\begin{eqnarray}}
\newcommand{\eea}{\end{eqnarray}}
\newcommand{\bna}{\begin{eqnarray*}}
\newcommand{\ena}{\end{eqnarray*}}

\numberwithin{equation}{section}

\theoremstyle{plain}
\newtheorem{lemma}{Lemma}[section]
\newtheorem{theorem}[lemma]{Theorem}
\newtheorem{corollary}[lemma]{Corollary}

\theoremstyle{definition}

\newtheorem{remark}{Remark}

\begin{document}
	
\title{Analytic twists of $\rm GL_2\times\rm GL_2$ automorphic forms}
		
\author{Bingrong Huang}
\address{Data Science Institute and School of Mathematics, Shandong University \\ Jinan,
Shandong 250100, China}
\email{brhuang@sdu.edu.cn}

\author{Qingfeng Sun}
\address{School of Mathematics and Statistics, Shandong University,
Weihai\\Weihai, Shandong 264209, China}
\email{qfsun@sdu.edu.cn}

\author{Huimin Zhang}
\address{Data Science Institute and School of Mathematics, Shandong University \\
Jinan, Shandong 250100, China}
\email{hmzhang@mail.sdu.edu.cn}

\date{}

\begin{abstract}
Let $f$ and $g$ be holomorphic or Maass cusp forms for $\rm SL_2(\mathbb{Z})$
with normalized Fourier coefficients $\lambda_f(n)$ and $\lambda_g(n)$, respectively. In this paper,
we prove nontrivial estimates for the sum
\bna
\sum_{n=1}^{\infty}\lambda_f(n)
\lambda_g(n)e\left(t \varphi\left(\frac{n}{X}\right)\right)V\left(\frac{n}{X}\right),
\ena
where $e(x)=e^{2\pi ix}$,
$V(x)\in \mathcal{C}_c^{\infty}(1,2)$, $t\geq 1$ is a large parameter and
$\varphi(x)$ is some nonlinear real valued smooth function.
Applications of these estimates include a subconvex
bound for the Rankin-Selberg $L$-function $L(s,f\otimes g)$ in the $t$-aspect,
an improved estimate for a nonlinear exponential twisted sum
and the following asymptotic formula for
the sum of the Fourier coefficients of certain $\rm{GL}_5$ Eisenstein series
$$
  \sum_{n \leq X}\lambda_{1\boxplus(f\times g)}(n)
  =L(1,f\times g)X + O(X^{\frac{2}{3}-\frac{1}{356}+\varepsilon})
$$
for any $\varepsilon>0$.

\end{abstract}
\thanks{(B. Huang is partially supported by the Young Taishan Scholars
Program of Shandong Province (Grant No. tsqn201909046),
Qilu Young Scholar Program of Shandong University,
and NSFC (Grant Nos. 12001314 and 12031008). Q. Sun is partially
  supported by the National Natural Science Foundation
  of China (Grant Nos. 11871306 and 12031008)}
	
	
	\maketitle
	
\section{Introduction}\label{introduction}

When studying number theory problems, one often runs into nonlinear exponential sums of the form
\bna
\sum_{n=1}^{\infty}a_n
e\left(t \varphi\left(\frac{n}{X}\right)\right)V\left(\frac{n}{X}\right),
\ena
where $a_n$ is some arithmetic function, here and throughout the paper, $e(x)=e^{2\pi ix}$,
$V(x)\in \mathcal{C}_c^{\infty}(1,2)$ is a smooth function with
support contained in $(1,2)$, $t,X\geq 1$ are large parameters and
$\varphi(x)$ is some nonlinear real valued smooth function. For example,
for an automorphic $L$-function $L(s,F)$, the subconvexity problem of $L(s,F)$ in the $t$-aspect
boils down to a nontrivial estimate for this sum with
$a_n=\lambda_F(n)$ being the Fourier coefficients of the automorphic form $F$ and
$\varphi(x)=-(\log x)/2\pi$. Here we remind that for $a_n$ $(n\sim X)$ satisfying
$\|a_n\|^2=\sum_n|a_n|^2\ll X$, the trivial bound
of this nonlinear exponential sum is $O(X)$. On the other hand, it is worth noting that the
square-root cancellation phenomenon should not
hold in general, as first found by Iwaniec, Luo
and Sarnak \cite{ILS} (see Appendix C, (C.17) and (C.18)) that
\bea\label{GL2}
\sum_{n=1}^{\infty} \lambda_F(n)e(-2\sqrt{qn})V\left(\frac{n}{X}\right)
=\frac{\lambda_F(q)}{q^{1/4}}\hat{V}(0)X^{3/4}+O\big((qX)^{1/4+\varepsilon}\big),
\eea
for any positive integer $q$ and any $\varepsilon>0$, where $\lambda_F(n)$ are
the normalized Fourier coefficients of
a $\rm SL_2(\mathbb{Z})$ holomorphic cusp form $F$ of weight $\kappa$ and
$\hat{V}(0)=2^{-1}i^\kappa (1-i)\int_0^{\infty}V(x)x^{-1/4}\mathrm{d}x$.
Moreover, Kaczorowski and Perelli \cite{KP} improved and extended this result for Selberg class
and this was later revisited by Ren and Ye \cite{Ren-Ye} for
$\rm GL_m$ Maass cusp forms.

For $a_n=\lambda_F(n)$ being the Fourier coefficients of an automorphic form $F$,
a natural way to study the associated nonlinear exponential twisted sum is to directly use the
functional equation of the automorphic $L$-function $L(s,F)$ or equivalently,
the Voronoi formula for $\lambda_F(n)$, as shown in \cite{KP} and \cite{Ren-Ye}.
However, if the nonlinear exponential function $e\left(t \varphi\left(n/X\right)\right)$
oscillates strong enough, there is a chance to get more savings by separating
the oscillations of $\lambda_F(n)$ and $e\left(t \varphi\left(n/X\right)\right)$ using
the $\delta$-method.  Kumar, Mallesham and Singh \cite{KMS19} first implemented this idea for
$\rm GL_3$ Maass cusp forms by
using the Duke-Friedlander-Iwaniec $\delta$-method given in \cite{IK} together with
the conductor-lowering trick due to Munshi \cite{Mun1},
and proved that for $t=X^{\beta}$ and $\varphi(x)=\alpha x^{\beta}$
$(\alpha\in \mathbb{R}\backslash\{0\}, 0<\beta<1)$
\bna
\sum_{n=1}^{\infty}\lambda_{\pi}(1,n)e\left(t\varphi\bigg(\frac{n}{X}\bigg)\right)
V\bigg(\frac{n}{X}\bigg)\ll_{\pi,\alpha,\beta} t^{3/10}X^{3/4+\varepsilon},
\ena
which improved the estimate $O(X^{3\beta/2}\log X)$ by Ren and Ye \cite{Ren-Ye-1}
for $\beta>5/8$. Here $\lambda_{\pi}(1,n)$ are the normalized Fourier coefficients
of a Hecke-Maass cusp form $\pi$ for $\rm GL_3(\mathbb{Z})$. See also the first author \cite{HB}.
For cusp forms on $\rm GL_2$, the associated nonlinear exponential twisted
sums were studied in Aggarwal, Holowinsky, Lin and Qi \cite{AHLQ} by a Bessel $\delta$-method.
Recently, Lin and the second author \cite{LS} studied the $\rm GL_3\times\rm GL_2$ case by
using the Duke-Friedlander-Iwaniec $\delta$-method in \cite{IK},
but unlike \cite{KMS19} without the conductor-lowering trick (as in Aggarwal \cite{Agg}).

The goal of this paper is to study nonlinear exponential twists of $\rm GL_2\times\rm GL_2$ automorphic forms.
More precisely, let $f$ and $g$ be either holomorphic or Maass cusp forms
for $\rm SL_2(\mathbb{Z})$
with normalized Fourier coefficients $\lambda_f(n)$ and $\lambda_g(n)$, respectively. Define
\bea\label{natural-sum}
S(X,t)=\sum_{n=1}^{\infty}\lambda_f(n)
\lambda_g(n)e\left(t \varphi\left(\frac{n}{X}\right)\right)V\left(\frac{n}{X}\right).
\eea
Our main result states as follows.
\begin{theorem}\label{main-theorem}
Let $\varphi(x)=\alpha\log x$ or $\alpha x^{\beta}$ ($\beta\in (0,1)\backslash \{1/2,3/4\}$, $\alpha\in \mathbb{R}\backslash \{0\}$).
Let $V(x)\in \mathcal{C}_c^{\infty}(1,2)$ with total variation $\rm{Var}(V)\ll 1$
and satisfying the condition
\begin{equation}\label{derivative-of-V}
V^{(j)}(x)\ll_j \triangle^j
\end{equation} for any integer $j\geq 0$ with $\triangle\ll t^{1/2-\varepsilon}$ for any $\varepsilon>0$.
Then we have
\bna
S(X,t) \ll_{f,g,\varphi,\varepsilon}t^{2/5}X^{3/4+\varepsilon}
\ena
for $t^{8/5}<X<t^{12/5}$.

\end{theorem}
\begin{remark}
The assumption $\triangle\ll t^{1/2-\varepsilon}$ arises when we use stationary phase
analysis for certain oscillatory integral in the proof (see \eqref{assumption-on-Delta}).
For the sake of simplicity, we have restricted $f$ and $g$ to be on the full modular group.
In fact, Theorem 1 can be similarly extended to modular forms of arbitrary level and nebentypus
without taking much effort.
\end{remark}

Since the test function $V$ in Theorem 1.1 allows oscillations,
we can remove it from the sum.
\begin{corollary}\label{sharp-cut-sum}
Same notation and assumptions as in Theorem \ref{main-theorem}. We have
\bna
\sum_{X<n\leq 2X}\lambda_f(n)
\lambda_g(n)e\left(t \varphi\left(\frac{n}{X}\right)\right)
\ll_{f,g,\varphi,\varepsilon} t^{2/5}X^{3/4+\varepsilon}
\ena
for $t^{8/5}< X<t^{12/5}$.
\end{corollary}

A special case of Theorem 1.1 is that $t=X^{\beta}$ and $\varphi(x)=\alpha x^{\beta}$
$(\alpha\in \mathbb{R}, 0<\beta<1, \beta\neq 1/2,3/4)$. Then Corollary 1.2 implies

\begin{corollary}
For any $\alpha\in \mathbb{R}\backslash\{0\}$, we have
\bna
\sum_{n\leq X}\lambda_f(n)
\lambda_g(n)e\big(\alpha n^{\beta}\big)
\ll_{f,g,\alpha,\beta,\varepsilon} X^{3/4+2\beta/5+\varepsilon}
\ena
for $5/12<\beta< 5/8$, $\beta\neq 1/2$.
\end{corollary}

For $15/32<\beta< 5/8$, $\beta\neq 1/2$,
Corollary 1.3 improves the estimate $O_{f,g,\alpha,\beta,\varepsilon}(X^{2\beta})$
by Czarnecki \cite{C}.

Theorem \ref{main-theorem} also admits an application in bounding
Rankin-Selberg $L$-functions on the critical line. We recall
\bna
L\left(s,f\otimes g\right)=\zeta(2s)\sum_{n=1}^{\infty}
\frac{\lambda_f(n)\lambda_g(n)}{n^{s}}
\ena
for $\mbox{Re}\,s>1$.
The convexity bound in the $t$-aspect is
$L\left(1/2+it,f\otimes g\right)\ll t^{1+\varepsilon}$ and recently
Acharya, Sharma and Singh \cite{ASS20} proved the subconvexity bound
$O_{f,g,\varepsilon} (t^{1-1/16+\varepsilon})$ by
using the Duke-Friedlander-Iwaniec $\delta$-method given in \cite{IK} together with
the conductor-lowering trick due to Munshi \cite{Mun1}. An application of the approximate
functional equation implies
\bna
L\left(\frac{1}{2}+it,f\otimes g\right)\ll \sup_{N\ll t^{2+\varepsilon}}\frac{1}{\sqrt{N}}
\left|\sum_{n=1}^{\infty}\lambda_f(n)
\lambda_g(n)n^{-it}V\left(\frac{n}{N}\right)\right|+ t^{-100}.
\ena
We demonstrate that the conductor-lowering trick in Acharya, Sharma and Singh's proof
can be removed and
applying Theorem 1.1 with $\varphi(x)=-(\log x)/2\pi$,
we improve the result of Acharya, Sharma and Singh.
\begin{corollary}\label{subconvexity} We have
\bna
L\left(1/2+it,f\otimes g\right)\ll_{f,g,\varepsilon} (1+|t|)^{9/10+\varepsilon}.
\ena
\end{corollary}

The best record bound for $L\left(1/2+it,f\otimes g\right)$ is the Weyl type bound
$L\left(1/2+it,f\otimes g\right)\ll (1+|t|)^{2/3+\varepsilon}$ due to
Blomer, Jana and Nelson \cite{BJN} by combining in
a substantial way representation theory, local harmonic analysis, and analytic number theory.
Bernstein and Reznikov showed the bound $(1+|t|)^{5/6+\varepsilon}$ in \cite{BR} (see Remarks 7.2.2.2).

Now we consider another application of Theorem \ref{main-theorem}.
Let $L(s,F)$ be an $L$-function of degree $d$ with coefficients $\lambda_F(1)=1$, $\lambda_F(n)\in\mathbb{C}$.
It is a fundamental problem to prove an asymptotic formula for the sum
\bna
  \mathcal A(X,F) = \sum_{n\leq X} \lambda_F(n).
\ena
Let $(\mu_{1,F},\ldots, \mu_{d,F})$ be the Satake parameter of $F$ at $\infty$.
Assume $L_\infty(s,F)=\prod_{1\leq j\leq d}\Gamma_\mathbb{R}(s-\mu_{j,F})$ does
not have poles for $\mathrm{Re}(s)>1/2+1/d$,
where $\Gamma_\mathbb{R}(s)=\pi^{-s/2}\Gamma(s/2)$.
Under the Ramanujan-Petersson conjecture
$\lambda_F(n)\ll n^{\varepsilon}$, Friedlander and Iwaniec \cite{Fri-Iwa}
established the following identity which
relates $\mathcal A(X,F)$ to its dual sum $\mathcal B(X,N)$
\bea\label{FI-Functional-eq}
\mathcal A(X,F)=\mathrm{Res}_{s=1}\frac{L(s,F)}{s}X+c_F\, X^{\frac{d-1}{2d}}\mathcal B(X,N)
+O\left(N^{-\frac{1}{d}}X^{\frac{d-1}{d}+\varepsilon}\right),
\eea
where $c_F$ is some constant depending on the form $F$ only and
Let
\bna
\mathcal B(X,N)=\sum_{n\leq N}\overline{\lambda_F(n)}\, n^{-\frac{d+1}{2d}}
\cos(2\pi d\left(nX\right)^{1/d}).
\ena
In particular, by estimating the sum $B(X,N)$ trivially and choosing
$N=X^{(d-1)/(d+1)}$, Friedlander and Iwaniec showed that
\bea\label{trivial-one}
\mathcal A(X,F)=\mathrm{Res}_{s=1}\frac{L(s,F)}{s}X+O\left(X^{\frac{d-1}{d+1}+\varepsilon}\right)
\eea
for any $\varepsilon>0$.
For $\lambda_F(n)=\sum_{n_1n_2n_3=n}\chi_1(n_1)
\chi_2(n_2)\chi_3(n_3)$, $\chi_j$ being primitive Dirichlet characters,
Friedlander and Iwaniec \cite{Fri-Iwa} proved an asymptotic formula with the error term
$O(X^{1/2-1/150+\varepsilon})$.
Recently, for $F=1\boxplus f$ and $\lambda_F(n)=\sum_{\ell m=n}\lambda_f(m)$, where
$f$ is a holomorphic cusp form for $\rm SL_2(\mathbb{Z})$,
Huang, Lin and Wang \cite{HLW} proved an asymptotic formula with the error term
$O(X^{1/2-4/739+\varepsilon})$. For $F=1\boxplus \mathrm{sym}^2f$
and $\lambda_F(n)=\sum_{\ell^2 m=n}\lambda_f(m)^2$, where
$f$ is a Hecke-holomorphic or Hecke-Maass cusp form for $\rm SL_2(\mathbb{Z})$,
Huang \cite{HB} proved an asymptotic formula with the error term
$O(X^{3/5-1/560+\varepsilon})$.
Under the Ramanujan-Petersson conjecture Lin and the second author \cite{LS} considered the
$\rm GL_3\times \rm GL_2$ case and proved the bound $O(X^{5/7-1/364+\varepsilon})$
for $F=\pi\otimes f$, where $\pi$ is a Hecke--Maass cusp form
for $\rm SL_3(\mathbb{Z})$ and $f$ is a holomorphic or Maass cusp form
for $\rm SL_2(\mathbb{Z})$.

As an application of Therorem \ref{main-theorem}, we improve \eqref{trivial-one}
for $F$ being certain $\rm GL_5$ Eisenstein series, namely when
$F=1\boxplus(f \times g)$ and $L(s,F)=\zeta(s)L(s,f\times g)$.
For simplification, we consider the holomorphic case.
In fact, our argument holds also for Maass cusp forms under the Ramanujan-Petersson
conjecture. Now let $f$ and $g$ be holomorphic Hecke cusp forms for
$\mathrm{SL}_2(\mathbb{Z})$ of weight $k$, $\kappa$, with
$k \geq \kappa \geq 12$, with normalized Fourier coefficients
$\lambda_f(n)$ and $\lambda_g(n)$, respectively. For $\rm{Re}(s) > 1$, we define
$$
L(s,1\boxplus(f \times g))= \zeta(s)L(s,f \times g)
= \zeta(s)\zeta(2 s)\sum_{n=1}^{\infty}
\frac{\lambda_f(n)\lambda_g(n)}{n^s}=
\sum_{n=1}^{\infty}\frac{\lambda_{1\boxplus(f\times g)}(n)}{n^s},
$$
where
$\lambda_{1\boxplus(f\times g)}(n):= \sum_{lm^2r=n}\lambda_f(r)\lambda_g(r)$.
Note that \eqref{trivial-one} reads
 $$
  \sum_{n \leq X}\lambda_{1\boxplus(f\times g)}(n)
  = L(1,f\times g)X + O(X^{\frac{2}{3}+\varepsilon}).
  $$
  We shall prove the following result.
\begin{corollary}\label{cor:main}
  Let $f$ and $g$ be holomorphic Hecke cusp forms for
  $\mathrm{SL}_2(\mathbb{Z})$ of weight $k$ and $\kappa$ with
  $12 \leq \kappa \leq k$, with normalized Fourier coefficients
  $\lambda_f(n)$ and $\lambda_g(n)$, respectively. Assume $f \perp g$. Then we have
  $$
  \sum_{n \leq X}\lambda_{1\boxplus(f\times g)}(n)
  = L(1,f\times g)X + O(X^{\frac{2}{3}-\frac{1}{356} +\varepsilon})
  $$
for any $\varepsilon>0$.
\end{corollary}

\begin{remark}
  At the end of the proof of Corollary \ref{cor:main} we will use the exponent pair $(\frac{13}{194}+\varepsilon,\, \frac{76}{97}+\varepsilon)$ which is a consequence of Bourgain's exponential pair in \cite{Bourgain} and the A-process in the theory of exponential pairs.
  This is the best known exponent pair we find for our problem. We essentially need to choose an exponent pair $(p,q)$ to minimize $\frac{38+33p-28q}{58+48p-43q}$.
\end{remark}

\medskip
The paper is organized as follows. In Section \ref{sketch-of-proof}, we
provide a quick sketch and key steps of the proof. In
Section \ref{review-of-cuspform}, we review some basic materials of
automorphic forms on $ \rm GL_2$ and estimates on exponential integrals.
Sections \ref{details-of-proof}
and \ref{proofs-of-technical-lemma} give details of the proof for Theorem \ref{main-theorem}
and in Sections \ref{proofs-of-corollary} and  \ref{sec:Cor} we complete the proofs
for Corollaries \ref{sharp-cut-sum} and \ref{cor:main}, respectively.

\bigskip
\noindent
{\bf Notation.}
Throughout the paper, the letters $q$, $m$ and $n$, with or without subscript,
denote integers. The letters $\varepsilon$ and $A$ denote arbitrarily small and large
positive constants, respectively, not necessarily the same at different occurrences.
We use $A\asymp B$ to mean that $c_1B\leq |A|\leq c_2B$ for some positive constants $c_1$ and $c_2$. The symbol
$\ll_{a,b,c}$ denotes that the implied constant depends at most on $a$, $b$ and $c$, and
$q\sim C$ means $C<q\leq 2C$.

\section{Outline of the proof}\label{sketch-of-proof}

In this section, we provide a quick sketch of the proof for Theorem \ref{main-theorem}.
Suppose we are working with the following sum
\bna
\mathcal{S}=\sum_{n\sim X}\lambda_f(n)
\lambda_g(n)e\left(t\varphi\bigg(\frac{n}{X}\bigg)\right).
\ena

The first step is writing
\bna
\mathcal{S}=\sum_{n\sim X}\lambda_f(n)\sum_{m\sim X}\lambda_g(m)
e\left(t\varphi\bigg(\frac{m}{X}\bigg)\right)\delta(m-n,0),
\ena
and using the $\delta$-method to detect the
Kronecker delta symbol $\delta(m-n,0)$. As in \cite{LS}, we use
the Duke-Friedlander-Iwaniec's $\delta$-method \eqref{DFI's} to write
\bea\label{before-voronoi}
\mathcal{S}&=&\frac{1}{Q} \int_{-X^{\varepsilon}}^{X^{\varepsilon}}
\sum_{q\sim Q}\frac{1}{q}\;\sideset{}{^\star}\sum_{a\bmod{q}}\,
\sum_{n\sim X}\lambda_f(n)e\left(-\frac{na}{q}\right)
e\left(-\frac{n\zeta}{qQ}\right)\nonumber\\
&&\sum_{m\sim X}\lambda_g(m)e\left(\frac{ma}{q}\right)
e\left(t\varphi\bigg(\frac{m}{X}\bigg)+\frac{m\zeta}{qQ}\right)\mathrm{d}\zeta,
\eea
where the $\star$ in the sum over $a$ means that the sum is restricted to $(a,q)=1$.

Next, we use  the $\mathrm{GL}_2$ Voronoi summation formulas to dualize the $m$- and $n$-sums.
The $m$-sum can be transformed into the following
\bea\label{GL2dual}
&&\sum_{m\sim X}\lambda_g(m)e\left(\frac{ma}{q}\right)
e\left(t\varphi\bigg(\frac{m}{X}\bigg)+\frac{m\zeta}{qQ}\right)
V\left(\frac{m}{X}\right)\nonumber\\
&&\leftrightarrow \frac{X}{Qt^{1/2}}
\sum_{\pm}\sum_{m\sim Q^2t^2/X}\lambda_g(m)e\left(-\frac{m\bar{a}}{q}\right)
\Phi^{\pm}\left(m,q,\zeta\right),
\eea
where
 \bna
\Phi^{\pm}\left(m,q,\zeta\right)=\int_0^\infty V(y)y^{-1/4}
e\left(t\varphi(y)+\frac{\zeta Xy}{qQ}\pm\frac{2\sqrt{mXy}}{q}\right)\mathrm{d}y.
\ena
If we assume for example $\varphi'(x)>0$, then by integration by parts,
$\Phi^{+}\left(m,q,\zeta\right)\ll X^{-A}$, and we only need
to consider the minus sign contribution.
Similarly, for the $n$-sum, we have
\bea\label{GL3dual}
&&\sum_{n\sim X}\lambda_f(n)e\left(-\frac{na}{q}\right)
e\left(-\frac{n\zeta}{qQ}\right)U\left(\frac{n}{X}\right)\nonumber\\
&&\leftrightarrow
X^{1/2}\sum_{n\sim X/Q^2}\lambda_f(n)e\left(\frac{n\bar{a}}{q}\right)
\Psi^{+}\left(n,q,\zeta\right)+O_A(X^{-A}),
\eea
where
\bna
\Psi^{+}\left(n,q,\zeta\right)=\int_0^\infty U(y)y^{-1/4}
e\left(-\frac{\zeta Xy}{qQ}+\frac{2\sqrt{nXy}}{q}\right)\mathrm{d}y.
\ena

We perform a stationary phase argument to get (note that $n\sim X/Q^2$)
\bna
\Psi^{+}\left(n,q,\zeta\right)\asymp \frac{q^{1/2}}{(nX)^{1/4}}
e\left(\frac{nQ}{q\zeta}\right)
U^{\natural}\left(\frac{nQ^2}{X\zeta^2}\right)
\asymp \frac{Q}{X^{1/2}}
e\left(\frac{nQ}{q\zeta}\right)
U^{\natural}\left(\frac{nQ^2}{X\zeta^2}\right)
\ena
for some smooth compactly supported function $U^{\natural}(y)$.
Then by plugging the dual sums \eqref{GL2dual} and \eqref{GL3dual} back into \eqref{before-voronoi}
and switching the orders of integration over $\zeta$ and $y$, we roughly get
\bea\label{intermediateS(X)}
\begin{split}
\mathcal{S}\approx& \frac{X}{Q^2t^{1/2}}\sum_{q\sim Q}\,
\sum_{m\sim Q^2t^2/X}
\lambda_g(m)
\sum_{n\sim X/Q^2}
\lambda_f(n)S(m-n,0;q)\\
&\times\int_0^\infty V(y)y^{-1/4}
e\left(t\varphi(y)-\frac{2\sqrt{mXy}}{q}\right)\mathcal{K}(y;n,q)\, \mathrm{d}y
\end{split}
\eea
where
\bna
\mathcal{K}(y;n,q)=
\int_{-X^{\varepsilon}}^{X^{\varepsilon}}
U^{\natural}\left(\frac{nQ^2}{X\zeta^2}\right)
e\left(\frac{\zeta Xy}{qQ}+\frac{nQ}{q \zeta}\right)
\mathrm{d}\zeta.
\ena

We evaluate the integral $\mathcal{K}(y;n,q)$ using the stationary phase method
(note that $n\sim X/Q^2$)
\bna
\mathcal{K}(y;n,q)\asymp \frac{n^{1/4}q^{1/2}Q}{X^{3/4}}e\left(\frac{2\sqrt{nXy}}{q}\right)
F(y)\asymp \frac{Q}{X^{1/2}}e\left(\frac{2\sqrt{nXy}}{q}\right)
F(y)
\ena
for some smooth compactly supported function $F(y)$.
Hence putting things together and writing the Ramanujan sum
$S\left(m-n,0;q\right)$ as $\sum_{d|(m-n,q)}d\mu(q/d)$, $\mathcal{S}$ in \eqref{intermediateS(X)} is roughly
equal to
\bna
\begin{split}
& \frac{X^{1/2}}{Qt^{1/2}}\sum_{q\sim Q}\,\sum_{d|q}d\mu\left(\frac{q}{d}\right)
\sum_{m\sim Q^2t^2/X}
\lambda_g(m)
\sum_{n\sim X/Q^2\atop n\equiv m\bmod d}
\lambda_f(n)\mathfrak{I}(m,n,q),
\end{split}
\ena
where
\bea\label{I-integral0}
\mathfrak{I}(m,n,q)=\int_0^\infty \widetilde{V}(y)
e\left(t\varphi(y)+\frac{2\sqrt{nXy}}{q}-\frac{2\sqrt{mXy}}{q}\right)\, \mathrm{d}y
\eea
for some smooth compactly supported function $\widetilde{V}(y)$.
Assume $\varphi(x)=c\log x$ or $cx^{\beta}$ with $\beta\in (0,1), \beta\neq 1/2$.
We apply the stationary phase analysis to the integral $\mathfrak{I}(m,n,q)$
 to get
\bna
\mathfrak{I}(m,n,q)
\sim e\left(t\varphi(y_0^2)-Dy_0\right)\mathfrak{I}^*(m,n,q)
\ena
where $y_0=\left(ct/D\right)^{1/\beta}\asymp 1$ with $D=2q^{-1}(mX)^{1/2}$ and
\bna
\mathfrak{I}^*(m,n,q)
\asymp t^{-1/2}e\left(\frac{2y_0n^{1/2}X^{1/2}}{q}\right).
\ena

To prepare for an application of the Poisson summation in the $m$-variable,
we now apply the Cauchy-Schwarz inequality to smooth the $m$-sum
and put the $n$-sum inside the absolute value squared to get
\bna
\mathcal{S}&\ll& \frac{X^{1/2}}{Qt^{1/2}}\sum_{q\sim Q}\,\sum_{d|q}d
\bigg(\sum_{m\sim Q^2t^2/X}|\lambda_g(m)|^2\bigg)^{1/2}
\bigg(\sum_{m\sim Q^2t^2/X}\bigg|\sum_{n\sim X/Q^2\atop n\equiv m\bmod d}
\lambda_f(n)\mathfrak{I}^*(m,n,q)\bigg|^2\bigg)^{1/2}\\
&\ll&t^{1/2}\sum_{q\sim Q}\,\sum_{d|q}d
\bigg(\sum_{m\sim Q^2t^2/X}\bigg|\sum_{n\sim X/Q^2\atop n\equiv m\bmod d}
\lambda_f(n)\mathfrak{I}^*(m,n,q)\bigg|^2\bigg)^{1/2}.
\ena

\begin{remark}
If we open the absolute value squared, by the Rankin-Selberg estimate
for $\lambda_f(n)$ and the trivial estimate $\mathfrak{I}^*(m,n,q)\ll t^{-1/2}$,
the contribution from the diagonal term
$n=n'$ is given by
\begin{equation}\label{S-diagonal}
\begin{split}
\mathcal{S}_{\text{diag}}
\ll& t^{1/2}\sum_{q\sim Q}\,\sum_{d|q}d
\bigg(\sum_{m\sim Q^2t^2/X}\sum_{n\sim X/Q^2\atop n\equiv m\bmod d}
|\lambda_f(n)|^2|\mathfrak{I}^*(m,n,q)|^2\bigg)^{1/2}\\
\ll& Q^{3/2}t,
\end{split}
\end{equation}
which will be fine for our purpose (i.e., $S_{\text{diag}}=o(X)$) as long as $Q\ll (X/t)^{3/2}$.
\end{remark}

Note that the oscillation in the $m$-variable of $\mathfrak{I}^*(m,n,q)$ in \eqref{I-integral0}
 is of size $2y_0n^{1/2}X^{1/2}/q\approx X/Q^2$.
So opening the absolute value squared and applying the Poisson summation formula
in the $m$-variable, we have
\begin{equation*}
\begin{split}
\sum_{m\sim Q^2t^2/X\atop m\equiv n\bmod d}\mathfrak{I}^*(m,n,q)
\overline{\mathfrak{I}^*(m,n',q)}\leftrightarrow \frac{Q^2t^2}{dX}
\sum_{\tilde{m}\ll \frac{dX/Q^2}{Q^2t^2/X}}
\,\mathcal{H}\left(\frac{\tilde{m}Q^2t^2}{dX}\right),
\end{split}
\end{equation*}
where
\bea\label{correlation-integral}
\mathcal{H}(x)=\int_{\mathbb{R}}
\mathfrak{I}^*\left(Q^2t^2\xi/X,n,q\right)
\overline{\mathfrak{I}^*\left(Q^2t^2\xi/X,n',q\right)}
\, e\left(-x\xi\right)\mathrm{d}\xi.
\eea

The contribution to $\mathcal{S}$ from the zero-frequency $\tilde{m}=0$
will roughly correspond to the diagonal contribution $S_{\text{diag}}$ in \eqref{S-diagonal}.
For the non-zero frequencies from the terms with $\tilde{m}\neq 0$, we note that
by performing stationary phase analysis, when $|x|$ is ``large", the expected estimate for
the triple integral $\mathcal{H}(x)$ in \eqref{correlation-integral} is
\bea\label{expect}
\mathcal{H}(x)\ll t^{-1/2}\cdot t^{-1/2}\cdot |x|^{-1/2},
\eea
which comes from the square-root cancellation of the two inner integrals
and the square-root cancellation in the $\xi$-variable. Note that this estimate does not hold for
``small" $|x|$. In fact, for these exceptional cases
the ``trivial" bound $\mathcal{H}(x)\ll t^{-1/2}\cdot t^{-1/2}$ will suffice for our purpose.
(These are the content of Lemma \ref{integral:lemma}). We ignore these exceptions and
plug the expected estimate \eqref{expect} for  $\mathcal{H}(x)$
into $\mathcal{S}$. It turns out
that the non-zero frequencies contribution $S_{\text{off-diag}}$ from
$\tilde{m}\neq 0$ to $\mathcal{S}$ is given by
\bna
S_{\text{off-diag}}&\ll&t^{1/2}\sum_{q\sim Q}\,\sum_{d|q}d
\bigg(\sum_{n\sim X/Q^2}|\lambda_f(n)|^2
\sum_{n'\sim X/Q^2\atop n'\equiv n\bmod d}\frac{Q^2t^2}{dX}
\sum_{0\neq \widetilde{m}\ll dX^2/(Q^4t^2)}\frac{d^{1/2}X^{1/2}}{|\widetilde{m}|^{1/2}Qt^2}\bigg)^{1/2}\\
&\ll&\frac{X^{5/4}}{Q}+X^{3/4}Q^{1/2}\\
&\ll&\frac{X^{5/4}}{Q}
\ena
provided that $Q<X^{1/3}$.
Hence combining this with the diagonal contribution $S_{\text{diag}}$ in \eqref{S-diagonal}, we get
\bna
\mathcal{S}\ll  Q^{3/2}t+\frac{X^{5/4}}{Q}.
\ena
By choosing $Q=X^{1/2}/t^{2/5}$ we obtain $\mathcal{S}\ll t^{2/5}X^{3/4}$ provided that $X<t^{12/5}$,
which improves over the trivial bound $\mathcal{S}\ll X$ as long as $t^{8/5}\ll X$.

\section{Preliminaries}\label{review-of-cuspform}

First we recall some basic results on automorphic forms for $\mathrm{GL}_2$.

\subsection{Holomorphic cusp forms for $\mathrm{GL}_2$}

Let $f$ be a holomorphic cusp form of weight $\kappa$ for $\rm SL_2(\mathbb{Z})$
with Fourier expansion
\bna
f(z)=\sum_{n=1}^{\infty}\lambda_f(n)n^{(\kappa-1)/2}e(nz)
\ena
for $\mbox{Im}\,z>0$, normalized such that $\lambda_f(1)=1$.
By the Ramanujan-Petersson conjecture proved by Deligne \cite{Del},
we have
$
\lambda_f(n)\ll \tau(n)\ll n^{\varepsilon}
$
with $\tau(n)$ being the divisor function.

For $h(x)\in \mathcal{C}_c(0,\infty)$, we set
\bea\label{intgeral transform-1}
\Phi_h(x) =2\pi i^{\kappa} \int_0^{\infty} h(y) J_{\kappa-1}(4\pi\sqrt{xy})\mathrm{d}y,
\eea
where $J_{\kappa-1}$ is the usual $J$-Bessel function of order $\kappa-1$.
We have the following Voronoi summation formula (see \cite[Theorem A.4]{KMV}).

\begin{lemma}\label{voronoiGL2-holomorphic}
Let $q\in \mathbb{N}$ and $a\in \mathbb{Z}$ be such
that $(a,q)=1$. For $X>0$, we have
\bna\label{voronoi for holomorphic}
\sum_{n=1}^{\infty}\lambda_f(n)e\left(\frac{an}{q}\right)h\left(\frac{n}{X}\right)
=\frac{X}{q} \sum_{n=1}^{\infty}\lambda_f(n)
e\left(-\frac{\overline{a}n}{q}\right)\Phi_h\left(\frac{nX}{q^2}\right),
\ena
where $\overline{a}$ denotes
the multiplicative inverse of $a$ modulo $q$.
\end{lemma}

The function $\Phi_h(x)$ has the following asymptotic expansion
when $x\gg 1$ (see \cite{LS}, Lemma 3.2).

\begin{lemma}\label{voronoiGL2-holomorphic-asymptotic}
For any fixed integer $J\geq 1$ and $x\gg 1$, we have
\bna
\Phi_h(x)=x^{-1/4} \int_0^\infty h(y)y^{-1/4}
\sum_{j=0}^{J}
\frac{c_{j} e(2 \sqrt{xy})+d_{j} e(-2 \sqrt{xy})}
{(xy)^{j/2}}\mathrm{d}y
+O_{\kappa,J}\left(x^{-J/2-3/4}\right),
 \ena
where $c_{j}$ and $d_{j}$ are constants depending on $\kappa$.
\end{lemma}

\subsection{Maass cusp forms for $\mathrm{GL}_2$}

Let $f$ be a Hecke-Maass cusp form for $\rm SL_2(\mathbb{Z})$
with Laplace eigenvalue $1/4+\mu^2$. Then $f$ has a Fourier expansion
$$
f(z)=\sqrt{y}\sum_{n\neq 0}\lambda_f(n)K_{i\mu}(2\pi |n|y)e(nx),
$$
where $K_{i\mu}$ is the modified Bessel function of the third kind.
The Fourier coefficients satisfy
\bea\label{individual bound}
\lambda_f(n)\ll n^{\vartheta},
\eea
where, here and throughout the paper, $\theta$ denotes the exponent towards the Ramanujan
conjecture for $\rm GL_2$ Maass forms. The Ramanujan conjecture states that $\vartheta=0$ and
the current record due to Kim and Sarnak \cite{KS} is $\vartheta=7/64$.
We also need the following
average bound (see for instance \cite[Lemma 1]{Murty})
\bea\label{GL2: Rankin Selberg}
\sum_{n\leq X}|\lambda_f(n)|^2= c_{f} X+O\big(X^{3/5}\big).
\eea

For $h(x)\in \mathcal{C}_c^{\infty}(0,\infty)$, we define the integral transforms
\bea\label{intgeral transform-2}
\begin{split}
\Phi_h^+(x) =& \frac{-\pi}{\sin(\pi i\mu)} \int_0^\infty h(y)\left(J_{2i\mu}(4\pi\sqrt{xy})
- J_{-2i\mu}(4\pi\sqrt{xy})\right) \mathrm{d}y,\\
\Phi_h^-(x) =& 4\varepsilon_f\cosh(\pi \mu)\int_0^\infty h(y)K_{2i\mu}(4\pi\sqrt{xy}) \mathrm{d}y,
\end{split}\eea
where $\varepsilon_f$ is an eigenvalue under the reflection operator.
We have the following Voronoi summation formula (see \cite[Theorem A.4]{KMV}).

\begin{lemma}\label{voronoiGL2-Maass}
Let $q\in \mathbb{N}$ and $a\in \mathbb{Z}$ be such
that $(a,q)=1$. For $X>0$, we have
\bna\label{voronoi for Maass form}
\sum_{n=1}^{\infty}\lambda_f(n)e\left(\frac{an}{q}\right)h\left(\frac{n}{X}\right)
= \frac{X}{q} \sum_{\pm}\sum_{n=1}^{\infty}\lambda_f(n)
e\left(\mp\frac{\overline{a}n}{q}\right)\Phi_h^{\pm}\left(\frac{nX}{q^2}\right),
\ena
where $\overline{a}$ denotes
the multiplicative inverse of $a$ modulo $q$.
\end{lemma}

For $x\gg 1$, we have (see (3.8) in \cite{LS})
\bea\label{The $-$ case}
\Phi_h^-(x)\ll_{\mu,A}x^{-A}.
\eea
For $\Phi_h^+(x)$ and $x\gg 1$, we have a similar asymptotic formula as for
$\Phi_h(x)$ in the holomorphic case (see \cite{LS}, Lemma 3.4).
\begin{lemma}\label{voronoiGL2-Maass-asymptotic}
For any fixed integer $J\geq 1$ and $x\gg 1$, we have
\bna
\Phi_h^{+}(x)=x^{-1/4} \int_0^\infty h(y)y^{-1/4}
\sum_{j=0}^{J}
\frac{c_{j} e(2 \sqrt{xy})+d_{j} e(-2 \sqrt{xy})}
{(xy)^{j/2}}\mathrm{d}y
+O_{\mu,J}\left(x^{-J/2-3/4}\right),
 \ena
where $c_{j}$ and $d_{j}$ are some constants depending on $\mu$.
\end{lemma}

\begin{remark}\label{decay-of-largeX}
For $x\gg X^{\varepsilon}$, we can choose $J$ sufficiently large so that
the contribution from the $O$-terms in Lemmas \ref{voronoiGL2-holomorphic-asymptotic} and
\ref{voronoiGL2-Maass-asymptotic}
is negligible. For the main terms
we only need to analyze the leading term $j=1$, as the analysis of the remaining
lower order terms is the same and their contribution is smaller
compared to that of the leading term.
\end{remark}

\subsection{Estimates for exponential integrals}

Let
\begin{equation*}
 I = \int_{\mathbb{R}} w(y) e^{i \varrho(y)} dy.
\end{equation*}
Firstly, we have the following estimates for exponential integrals
(see \cite[Lemma 8.1]{BKY}  and \cite[Lemma A.1]{AHLQ}).
	
	\begin{lemma}\label{lem: upper bound}
		Let $w(x)$ be a smooth function    supported on $[ a, b]$ and
        $\varrho(x)$ be a real smooth function on  $[a, b]$. Suppose that there
		are   parameters $Q, U,   Y, Z,  R > 0$ such that
		\begin{align*}
		\varrho^{(i)} (x) \ll_i Y / Q^{i}, \qquad w^{(j)} (x) \ll_{j } Z / U^{j},
		\end{align*}
		for  $i \geqslant 2$ and $j \geqslant 0$, and
		\begin{align*}
		| \varrho' (x) | \geqslant R.
		\end{align*}
		Then for any $A \geqslant 0$ we have
		\begin{align*}
		I \ll_{ A} (b - a)
Z \bigg( \frac {Y} {R^2Q^2} + \frac 1 {RQ} + \frac 1 {RU} \bigg)^A .
		\end{align*}
			\end{lemma}

Next, we need the following evaluation for exponential integrals
which are
 Lemma 8.1 and Proposition 8.2 of \cite{BKY} in the language of inert functions
 (see \cite[Lemma 3.1]{KPY}).

Let $\mathcal{F}$ be an index set, $Y: \mathcal{F}\rightarrow\mathbb{R}_{\geq 1}$ and under this map
$T\mapsto Y_T$
be a function of $T \in \mathcal{F}$.
A family $\{w_T\}_{T\in \mathcal{F}}$ of smooth
functions supported on a product of dyadic intervals in $\mathbb{R}_{>0}^d$
is called $Y$-inert if for each $j=(j_1,\ldots,j_d) \in \mathbb{Z}_{\geq 0}^d$
we have
\bna
C(j_1,\ldots,j_d)
= \sup_{T \in \mathcal{F} } \sup_{(y_1, \ldots, y_d) \in \mathbb{R}_{>0}^d}
Y_T^{-j_1- \cdots -j_d}\left| y_1^{j_1} \cdots y_d^{j_d}
w_T^{(j_1,\ldots,j_d)}(y_1,\ldots,y_d) \right| < \infty.
\ena

\begin{lemma}
\label{lemma:exponentialintegral}
 Suppose that $w = w_T(y)$ is a family of $Y$-inert functions,
 with compact support on $[Z, 2Z]$, so that
$w^{(j)}(y) \ll (Z/Y)^{-j}$.  Also suppose that $\varrho$ is
smooth and satisfies $\varrho^{(j)}(y) \ll H/Z^j$ for some
$H/X^2 \geq R \geq 1$ and all $y$ in the support of $w$.
\begin{enumerate}
 \item
 If $|\varrho'(y)| \gg H/Z$ for all $y$ in the support of $w$, then
 $I \ll_A Z R^{-A}$ for $A$ arbitrarily large.
 \item If $\varrho''(y) \gg H/Z^2$ for all $y$ in the support of $w$,
 and there exists $y_0 \in \mathbb{R}$ such that $\varrho'(y_0) = 0$ (note $y_0$ is
 necessarily unique), then
 \begin{equation}
  I = \frac{e^{i \varrho(y_0)}}{\sqrt{\varrho''(y_0)}}
 F(y_0) + O_{A}(  Z R^{-A}),
 \end{equation}
where $F(y_0)$ is an $Y$-inert function (depending on $A$)  supported
on $y_0 \asymp Z$.
\end{enumerate}
\end{lemma}

We also need the
second derivative test (see \cite[Lemma 5.1.3]{Hux2}).
	
\begin{lemma}\label{lem: 2st derivative test, dim 1}
Let $\varrho(x)$ be real and twice
differentiable on the open interval $[a, b]$
with $ \varrho'' (x) \gg \lambda_0>0$  on $[a, b]$. Let $w(x)$
be real on $[ a, b]$ and let $V_0$ be its total
variation on $[ a, b]$ plus the maximum modulus of $w(x)$ on $[ a, b]$.
Then
		\begin{align*}
	I\ll \frac {V_0} {\sqrt{\lambda_0}}.
		\end{align*}
	\end{lemma}

\section{Proof of the main theorem}\label{details-of-proof}

In this section, we provide the details of the proof for Theorem \ref{main-theorem}.
Recall
\bea\label{main sum}
S(X,t)=\sum_{n=1}^{\infty}\lambda_f(n)
\lambda_g(n)e\left(t \varphi\left(\frac{n}{X}\right)\right)V\left(\frac{n}{X}\right),
\eea
where
$V(x)\in \mathcal{C}_c^{\infty}(1,2)$ with total variation $\text{Var}(V)\ll 1$ and
satisfying \eqref{derivative-of-V} that
$V^{(j)}(x)\ll_j \triangle^j$ for any integer $j\geq 0$ with $\triangle\ll t^{1/2-\varepsilon}$.
Without loss of generality, we assume that the function $\varphi$ satisfies
\bea\label{first-derivative-varphi}
\varphi'(x)>0, \qquad \varphi''(x)\gg 1.
\eea
(The case $\varphi'(x)<0$ can be analyzed analogously.)

\subsection{Applying DFI's $\delta$-method}

Define $\delta: \mathbb{Z}\rightarrow \{0,1\}$ with
$\delta(0)=1$ and $\delta(n)=0$ for $n\neq 0$.
As in \cite{LS}, we will use a version of the circle method of
Duke, Friedlander and Iwaniec (see \cite[Chapter 20]{IK})
which states that for any $n\in \mathbb{Z}$ and $Q\in \mathbb{R}^+$, we have
\bea\label{DFI's}
\delta(n)=\frac{1}{Q}\sum_{q\sim Q} \;\frac{1}{q}\;
\sideset{}{^\star}\sum_{a\bmod{q}}e\left(\frac{na}{q}\right)
\int_\mathbb{R}g(q,\zeta) e\left(\frac{n\zeta}{qQ}\right)\mathrm{d}\zeta
\eea
where the $\star$ on the sum indicates
that the sum over $a$ is restricted to $(a,q)=1$.
The function $g$ has the following properties (see (20.158) and (20.159)
of \cite{IK} and Lemma 15 of \cite{HB})
\bea\label{g-h}
g(q,\zeta)\ll |\zeta|^{-A}, \;\;\;\;\;\; g(q,\zeta) =1+h(q,\zeta)\;\;\text{with}\;\;h(q,\zeta)=
O\left(\frac{Q}{q}\left(\frac{q}{Q}+|\zeta|\right)^A\right)
\eea
for any $A>1$ and
\bea\label{rapid decay g}
\zeta^j\frac{\partial^j}{\partial \zeta^j}g(q,\zeta)\ll (\log Q)\min\left\{\frac{Q}{q},\frac{1}{|\zeta|}\right\}, \qquad j\geq 1.
\eea
In particular the first property in \eqref{g-h} implies that
the effective range of the integration in
\eqref{DFI's} is $[-X^\varepsilon, X^\varepsilon]$.

We write \eqref{main sum} as
\bna
S(X,t)=\sum_{n=1}^{\infty}\lambda_{f}(n)
U\left(\frac{n}{X}\right)\sum_{m=1}^{\infty}
\lambda_g(m)e\left(t \varphi\left(\frac{m}{X}\right)\right)
V\left(\frac{m}{X}\right)\delta(m-n),
\ena
where $U(x)\in \mathcal{C}_c^{\infty}(1/2,5/2)$ satisfying $U(x)=1$
for $x\in [1,2]$ and $U^{(j)}(x)\ll_j 1$ for any integer $j\geq 0$.
Plugging the identity \eqref{DFI's} for
$\delta(m-n)$ in and
exchanging the order of integration and summations,
we get
\bna
S(X,t)&=&\frac{1}{Q}
\int_{\mathbb{R}} \sum_{q\sim Q}\frac{g(q,\zeta)}{q}\;
\sideset{}{^\star}\sum_{a\bmod{q}}
\left\{\sum_{n=1}^{\infty}\lambda_{f}(n)
e\left(-\frac{na}{q}\right)U\left(\frac{n}{X}\right)
e\left(-\frac{n\zeta}{qQ}\right)\right\}\nonumber\\
&& \left\{\sum_{m=1}^{\infty}\lambda_{g}(m)e\left(\frac{ma}{q}\right)
V\left(\frac{m}{X}\right)
e\left(t \varphi\left(\frac{m}{X}\right)+\frac{m\zeta}{qQ}\right)\right\}\mathrm{d}\zeta.
\ena
Note that the contribution from $|\zeta|\leq X^{-B}$ is negligible for $B>0$ sufficiently large.
Moreover, by the first property in \eqref{g-h}, we can restrict $\zeta$ in the range
$|\zeta|\leq X^{\varepsilon}$ up to an negligible error. So we can
insert a smooth partition of unity for the $\zeta$-integral and get
\bna
S(X,t)&=&\sum_{X^{-B}\ll \Xi\ll X^{\varepsilon}\atop \text{dyadic}}\frac{1}{Q}
\int_{\mathbb{R}}W\left(\frac{\zeta}{\Xi}\right) \sum_{q\sim Q}\frac{g(q,\zeta)}{q}\;
\sideset{}{^\star}\sum_{a\bmod{q}}
\left\{\sum_{n=1}^{\infty}\lambda_f(n)
e\left(-\frac{na}{q}\right)U\left(\frac{n}{X}\right)
e\left(-\frac{n\zeta}{qQ}\right)\right\}\nonumber\\
&& \left\{\sum_{m=1}^{\infty}\lambda_g(m)e\left(\frac{ma}{q}\right)
V\left(\frac{m}{X}\right)
e\left(t \varphi\left(\frac{m}{X}\right)+\frac{m\zeta}{qQ}\right)\right\}\mathrm{d}\zeta
+O_A(X^{-A}),
\ena
where $\widetilde{W}(x)\in \mathcal{C}_c^{\infty}(1,2)$,
 satisfying $\widetilde{W}^{(j)}(x)\ll_j 1$ for any integer $j\geq 0$.

Next we break the $q$-sum $\sum_{q\sim Q}$ into dyadic segments $q\sim C$ with $1\ll C\ll Q$ and write
\bea\label{C range}
S(X,t)=\sum_{X^{-B}\ll \Xi\ll X^{\varepsilon}}
\sum_{1\ll C\ll Q\atop \text{dyadic}}\mathscr{S}(C,\Xi)+O(X^{-A}),
\eea
where $\mathscr{S}(C,\Xi)=\mathscr{S}(X,t,C,\Xi)$ is
\bea\label{beforeVoronoi}
\mathscr{S}(C,\Xi)&=&\frac{1}{Q}
\int_{\mathbb{R}}W\left(\frac{\zeta}{\Xi}\right) \sum_{q\sim C}\frac{g(q,\zeta)}{q}\;
\sideset{}{^\star}\sum_{a\bmod{q}}
\left\{\sum_{n=1}^{\infty}\lambda_f(n)
e\left(-\frac{na}{q}\right)U\left(\frac{n}{X}\right)
e\left(-\frac{n\zeta}{qQ}\right)\right\}\nonumber\\
&& \left\{\sum_{m=1}^{\infty}\lambda_g(m)e\left(\frac{ma}{q}\right)
V\left(\frac{m}{X}\right)
e\left(t \varphi\left(\frac{m}{X}\right)+\frac{m\zeta}{qQ}\right)\right\}\mathrm{d}\zeta.
\eea

Without loss of generality, for the $\zeta$-integral,
we only consider the contribution from $\zeta\geq 0$ (the contribution from $\zeta\leq 0$ can be
estimated similarly). By abuse of notation, we still write the contribution from $\zeta\geq 0$
as $\mathscr{S}(C,\Xi)$. We now proceed to estimate $\mathscr{S}(C,\Xi)$.

\subsection{Voronoi summations}
In what follows, we dualize the $n$-and $m$-sums in \eqref{beforeVoronoi} using
Voronoi summation formulas.

We first consider the sum over $m$.
Depending on whether $f$ is holomorphic or Maass, we apply
Lemma \ref{voronoiGL2-holomorphic}
or Lemma \ref{voronoiGL2-Maass} respectively with $h_1(y)=V(y)
e\left(t\varphi(y)+\zeta Xy/qQ\right)$, to transform the $m$-sum in
\eqref{beforeVoronoi} into
\bea\label{after GL2 Voronoi}
\frac{X}{q}\sum_{\pm}\sum_{m=1}^{\infty}\lambda_{g}(m)e\left(\mp \frac{m\overline{a}}{q}\right)
\Phi_{h_1}^{\pm}\left(\frac{mX}{q^2}\right),
\eea
where if $g$ is holomorphic, $\Phi_{h_1}^+(x)=\Phi_{h_1}(x)$ with $\Phi_{h_1}(x)$
given by \eqref{intgeral transform-1} and $\Phi_{h_1}^-(x)=0$,
while for $g$ a Hecke--Maass cusp form,
$\Phi_{h_1}^{\pm}(x)$ are given by \eqref{intgeral transform-2}.

Similarly, we apply
Lemma \ref{voronoiGL2-holomorphic}
or Lemma \ref{voronoiGL2-Maass} with
$h_2(y)=U(y)e\left(-\zeta Xy/qQ\right)$ to transform te
$n$-sum in \eqref{beforeVoronoi} into
\bea\label{$n$-sum after GL2 Voronoi}
\frac{X}{q}\sum_{\pm}\sum_{n=1}^{\infty}\lambda_f(n)e\left(\pm \frac{n\overline{a}}{q}\right)
\Phi_{h_2}^{\pm}\left(\frac{nX}{q^2}\right),
\eea
where if $f$ is holomorphic, $\Phi_{h_2}^+(x)=\Phi_{h_2}(x)$ with $\Phi_{h_2}(x)$
given by \eqref{intgeral transform-1} and $\Phi_{h_2}^-(x)=0$,
while for $f$ a Hecke--Maass cusp form, $\Phi_{h_2}^{\pm}(x)$ are given by \eqref{intgeral transform-2}.

As is typical in applying the $\delta$-method, we assume that
\bea\label{assumption 1}
Q<X^{1/2-\varepsilon}.
\eea
Then we have $mX/q^2\gg X^{\varepsilon}$ and $nX/q^2\gg X^{\varepsilon}$. In particular,
by \eqref{The $-$ case}, the contributions from
$\Phi_{h_1}^{-}\left(mX/q^2\right)$ and $\Phi_{h_2}^{-}\left(nX/q^2\right)$ are negligible.
For $\Phi_{h_1}^{+}\left(mX/q^2\right)$ and $\Phi_{h_2}^{+}\left(nX/q^2\right)$, we apply
Lemma \ref{voronoiGL2-holomorphic-asymptotic},
Lemma \ref{voronoiGL2-Maass-asymptotic}
and Remark \ref{decay-of-largeX} and find that
the sum \eqref{after GL2 Voronoi} is asymptotically equal to
\bea\label{integral 1}
 \frac{X^{3/4}}{q^{1/2}}\sum_{\pm}\sum_{m=1}^{\infty}\frac{\lambda_g(m)}{m^{1/4}}
e\left(-\frac{m\overline{a}}{q}\right)\int_0^\infty V(y)y^{-1/4}
e\left(t\varphi(y)+\frac{\zeta Xy}{qQ}\pm \frac{2\sqrt{mXy}}{q}\right)\mathrm{d}y,
 \eea
and the sum \eqref{$n$-sum after GL2 Voronoi} is asymptotically equal to
 \bea\label{integral 2}
 \frac{X^{3/4}}{q^{1/2}}\sum_{\pm}\sum_{n=1}^{\infty}\frac{\lambda_f(n)}{n^{1/4}}
e\left(\frac{n\overline{a}}{q}\right)\int_0^\infty U(y)y^{-1/4}
e\left(-\frac{\zeta Xy}{qQ}\pm \frac{2\sqrt{nXy}}{q}\right)\mathrm{d}y.
 \eea
 Note that by the assumption \eqref{first-derivative-varphi},
 the first derivative of the phase function of the
 exponential function in \eqref{integral 1} in the plus case is
 \bna
 t\varphi'(y)+\frac{\zeta X}{qQ}+\frac{\sqrt{mX/y}}{q}\gg X^{\varepsilon}.
 \ena
 By applying integration by parts
 repeatedly, one finds that the contribution from the plus case is negligible.
 Similarly, the contribution from the minus case in \eqref{integral 2} is negligible.

Assembling the above results, $\mathscr{S}(C,\Xi)$ in \eqref{beforeVoronoi} is asymptotically equal to
 \bea\label{GL2}
&&
 \frac{X^{3/2}}{Q}\sum_{q\sim C}q^{-2}
\sum_{m=1}^{\infty}\frac{\lambda_g(m)}{m^{1/4}}
\sum_{n=1}^{\infty}\frac{\lambda_f(n)}{n^{1/4}}S\left(m-n,0;q\right)\mathcal{I}(m,n,q,\Xi),
 \eea
where
\bea\label{I definition}
\mathcal{I}(m,n,q,\Xi)=
\int_{0}^{\infty}W\left(\frac{\zeta}{\Xi}\right)
g(q,\zeta)
\Phi\left(m,q,\zeta\right)\Psi\left(n,q,\zeta\right)
\mathrm{d}\zeta
\eea
with
\bea\label{Phi definition}
\Phi\left(m,q,\zeta\right)=\int_0^\infty V(y)y^{-1/4}
e\left(t\varphi(y)+\frac{\zeta Xy}{qQ}-\frac{2\sqrt{mXy}}{q}\right)\mathrm{d}y
\eea
and
\bea\label{Psi definition}
\Psi\left(n,q,\zeta\right)=\int_0^\infty U(y)y^{-1/4}
e\left(-\frac{\zeta Xy}{qQ}+\frac{2\sqrt{nXy}}{q}\right)\mathrm{d}y.
\eea

Note that for $\triangle<t^{1/2-\varepsilon}$, defined in \eqref{derivative-of-V},
by Lemma \ref{lem: upper bound},
 the integral $\Phi\left(m,q,\zeta\right)$ is negligibly small unless
 $\sqrt{mX}/q\ll X^{\varepsilon}\max\left\{t,X\Xi/qQ\right\}$.
 Thus we only need to consider those $m$ in the range
 $m \ll X^{\varepsilon}\max\{C^2t^2/X,X\Xi^2/Q^2\}$.
Similarly, up to a negligible error, we only need to consider those $n$
in the range $n\asymp  X\Xi^2/Q^2$.
Making smooth partitions of unity into dyadic segments to the sums over $m$ and $n$ in \eqref{GL2},
we arrive at
\bea\label{M-N1-range}
\mathscr{S}(C,\Xi)\ll
\sum_{ M \ll X^{\varepsilon}\max\{C^2t^2/X,X\Xi^2/Q^2\}\atop \text{dyadic}}
\sum_{N_1\asymp
X\Xi^2/Q^2\atop \text{dyadic}}
\left|\mathbf{T}\right|,
\eea
where temporarily, $\mathbf{T}:=\mathbf{T}(X,C,M,N_1,\Xi)$ is given by
\bea\label{T definition}
\mathbf{T}=\frac{X^{3/2}}{Q}\sum_{q\sim C}
q^{-2}
\sum_{m\sim M }
\frac{\lambda_g(m)}{m^{1/4}}
\sum_{n\sim N_1}\frac{\lambda_f(n)}{n^{1/4}}S\left(m-n,0;q\right)\mathcal{I}(m,n,q,\Xi).
\eea

Now we consider the integral $\mathcal{I}(m,n,q,\Xi)$ in \eqref{I definition}.
First we apply the stationary phase method to the integral $\Psi\left(n,q,\zeta\right)$
in \eqref{Psi definition}.
The stationary point $y_0$ is given by $y_0=nQ^2/(X\zeta^2)$.
Applying Lemma \ref{lemma:exponentialintegral} (2) with $Y=Z=1$ and
$H=R=\sqrt{nX}/q\gg X^{\varepsilon}$, we obtain
\bna
\Psi\left(n,q,\zeta\right)
=\frac{q^{1/2}}{(nX)^{1/4}}
e\left(\frac{nQ}{q\zeta}\right)
U^{\natural}\left(\frac{nQ^2}{X\zeta^2}\right)+O_A\left(X^{-A}\right),
\ena
where $U^{\natural}$ is an $1$-inert function (depending on $A$)
supported on $y_0 \asymp 1$.
Plugging this asymptotic formula for $\Psi\left(n,q,\zeta\right)$ and \eqref{Phi definition}
into \eqref{I definition}
and switching the order of integrations, one has
\bea\label{I-middle}
\mathcal{I}(m,n,q,\Xi)=\frac{q^{1/2}}{(nX)^{1/4}}
\int_0^\infty \mathcal{K}(y;n,q,\Xi)V(y)y^{-1/4}
e\bigg(t\varphi(y)-\frac{2\sqrt{mXy}}{q}\bigg)
\mathrm{d}y+O_A(X^{-A})
\eea
with
\bna
\mathcal{K}(y;n,q,\Xi)=
\int_0^{\infty}g(q,\zeta)
W\bigg(\frac{\zeta}{X^{\varepsilon}}\bigg)\widetilde{W}\bigg(\frac{\zeta}{\Xi}\bigg)
U^{\natural}\bigg(\frac{nQ^2}{X\zeta^2}\bigg)
e\bigg(\frac{\zeta Xy}{qQ}+\frac{nQ}{q\zeta}\bigg)
\mathrm{d}\zeta.
\ena
Next, we derive an asymptotic expansion for $\mathcal{K}(y;n,q,\Xi)$.
By changing variable $nQ^2/(X\zeta^2)\rightarrow \zeta$,
\bna
\mathcal{K}(y;n,q,\Xi)
=\frac{n^{1/2}Q}{X^{1/2}}
\int_0^{\infty}\phi(\zeta)
\exp\big(i\varpi(\zeta)\big)\mathrm{d}\zeta,
\ena
where
\bna
\phi(\zeta):=-\frac{1}{2}\zeta^{-3/2}U^{\natural}(\zeta)
g\bigg(q,\frac{n^{1/2}Q}{\zeta^{1/2}X^{1/2}}\bigg)
W\bigg(\frac{n^{1/2}Q}{\zeta^{1/2}X^{1/2+\varepsilon}}\bigg)
\widetilde{W}\bigg(\frac{n^{1/2}Q}{\zeta^{1/2}X^{1/2}\Xi}\bigg)
\ena
and the phase function $\varpi(\zeta)$ is given by
\bna
\varpi(\zeta)=\frac{2\pi n^{1/2}X^{1/2}}{q}
\left(y\zeta^{-1/2}+\zeta^{1/2}\right).
\ena
Note that
\bna
\varpi'(\zeta)=
\frac{\pi n^{1/2}X^{1/2}}{q}
\left(-y\zeta^{-3/2}+\zeta^{-1/2}\right),
\ena
and for $j\geq 2$,
\bna
\varpi^{(j)}(\zeta)
=\left(-\frac{3}{2}\right)\cdot\cdot\cdot\left(\frac{1}{2}-j\right)
\frac{\pi n^{1/2}X^{1/2}}{q}\left(
-y\zeta^{-1/2-j}+\frac{1}{2j-1}\zeta^{1/2-j}\right).
\ena
Thus the stationary point is $\zeta_0=y$ and
$\varpi^{(j)}(\zeta)\ll_j n^{1/2}X^{1/2}/q$ for $j\geq 2$.
By \eqref{rapid decay g}, we have $\phi^{(j)}(\zeta)\ll_j X^{\varepsilon}$.
Applying Lemma \ref{lemma:exponentialintegral} (2) with $Y=Z=1$ and
$H=R=n^{1/2}X^{1/2}/q\gg X^{\varepsilon}$, we obtain
\bea\label{K-integral}
\mathcal{K}(y;n,q,\Xi)
=\frac{n^{1/4}q^{1/2}Q}{X^{3/4}}e\left(\frac{2\sqrt{nXy}}{q}\right)
F(y)+O_A\left(X^{-A}\right),
\eea
where $F(y)=F(y;\Xi)$ is an inert function (depending on $A$ and $\Xi$)
supported on $\zeta_0 \asymp 1$.
Substituting \eqref{K-integral} into \eqref{I-middle}, we get
\bea\label{I-middle-0}
\mathcal{I}(m,n,q,\Xi)=\frac{qQ}{X}
\int_0^\infty V(y)F(y)y^{-1/4}
e\bigg(t\varphi(y)+\frac{2\sqrt{nXy}}{q}-\frac{2\sqrt{mXy}}{q}\bigg)
\mathrm{d}y+O_A(X^{-A}).
\eea

Further substituting \eqref{I-middle-0} into \eqref{T definition} and writing the Ramanujan sum
$S\left(m-n,0;q\right)$ as $\sum_{d|(m-n,q)}d\mu(q/d)$, one has
\bea\label{beforeCauchy}
\mathbf{T}=X^{1/2}\sum_{q\sim C}
q^{-1}\sum_{d|q}d\mu\left(\frac{q}{d}\right)
\sum_{m\sim M }
\frac{\lambda_g(m)}{m^{1/4}}
\sum_{n\sim N_1\atop n\equiv m\bmod d}\frac{\lambda_f(n)}{n^{1/4}}\mathfrak{I}(m,n,q)+O_A\left(X^{-A}\right),
\eea
where $\mathfrak{I}(m,n,q)=\mathfrak{I}(m,n,q;\Xi)$ is given by
\bea\label{I-definition}
\mathfrak{I}(m,n,q)
=\int_0^\infty \widetilde{V}(y)
e\bigg(t\varphi(y)+\frac{2\sqrt{nXy}}{q}-\frac{2\sqrt{mXy}}{q}\bigg)
\mathrm{d}y.
\eea
Here $\widetilde{V}(y)=V(y)F(y)y^{-1/4}$
satisfying $\widetilde{V}^{(j)}(y)\ll_j \triangle^j$ and $\text{Var}(\widetilde{V})\ll 1$.
Recall $\triangle$ denotes the quantity such that $V^{(j)}(x)\ll \triangle^j$ (
see \eqref{derivative-of-V}).

Making a
change of variable  $y\rightarrow y^2$ in \eqref{I-definition}, one has
\bea\label{I-change-0}
\mathfrak{I}(m,n,q)
=2\int_0^\infty y\widetilde{V}(y^2)
e\left(t\varphi(y^2)+\frac{2X^{1/2}}{q}\left(n^{1/2}-m^{1/2}\right)y\right)
\mathrm{d}y.
\eea
Since the properties of the integral $\mathfrak{I}(m,n,q)$ depend on the size of $C$,
we split the modulus $C$ according to $C\leq X^{1+\varepsilon}\Xi/(Qt)$ or
$X^{1+\varepsilon}\Xi/(Qt) \leq C\ll Q$.

\subsection{The case of small modulus}
We first deal with the case $1\ll C\leq X^{1+\varepsilon}\Xi/(Qt)$.
If we assume $\left(\varphi(y^2)\right)''\gg 1$, equivalently $\varphi(y)\neq cy^{1/2}+c_0$ for any constant $c_0$,
then the second derivative of the phase function satisfies
\bna
t\left(\varphi(y^2)\right)''\gg t
\ena
and by Lemma \ref{lem: 2st derivative test, dim 1}, we have
\bna
\mathfrak{I}(m,n,q)\ll t^{-1/2}.
\ena
By this estimate, \eqref{individual bound} and \eqref{GL2: Rankin Selberg}, $\mathbf{T}$ in \eqref{beforeCauchy}
can be bounded by
\bna\label{small T}
\mathbf{T}&\ll& \frac{X^{1/2}N_1^{\vartheta}}{t^{1/2}(MN_1)^{1/4}}\sum_{q\sim C}
q^{-1}\sum_{d|q}d
\sum_{m\sim M }|\lambda_g(m)|
\sum_{n\sim N_1\atop n\equiv m\bmod d}1\nonumber\\
&\ll& \frac{X^{1/2}M^{3/4}N_1^{\vartheta}}{t^{1/2}N_1^{1/4}}\sum_{q\sim C}
q^{-1}\sum_{d|q}d\bigg(1+\frac{N_1}{d}\bigg)\nonumber\\
&\ll&t^{-1/2}X^{1/2}M^{3/4}N_1^{-1/4+\vartheta}(C+N_1)\nonumber\\
&\ll&t^{-1/2}X^{1/2+\varepsilon}\frac{X^{1/2+\vartheta}\Xi^{1+2\vartheta}}{Q^{1+2\vartheta}}
\left(\frac{X\Xi}{Qt}+\frac{X\Xi^2}{Q^2}\right)\nonumber\\
&\ll&\frac{X^{2+\vartheta+\varepsilon}}{Q^{2+2\vartheta}t^{1/2}}\left(\frac{1}{t}+\frac{1}{Q}\right)
\ena
recalling $\Xi\ll X^{\varepsilon}$,
$M \ll X^{\varepsilon}\max\{C^2t^2/X,X\Xi^2/Q^2\}\ll X^{1+\varepsilon}\Xi^2/Q^2$ and
$N_1\asymp  X\Xi^2/Q^2$ in \eqref{M-N1-range}.
Assuming
\bea\label{assumption 3}
Q<t
\eea
Then the contribution from $1\ll C\leq X^{1+\varepsilon}\Xi/(Qt)$ to $\mathscr{S}(C,\Xi)$ in
\eqref{M-N1-range} is
\bea\label{small contribution}
\frac{X^{2+\vartheta+\varepsilon}}{Q^{3+2\vartheta}t^{1/2}}.
\eea

\subsection{The case of large modulus}

In the subsequent sections, we deal with the case $X^{1+\varepsilon}\Xi/(Qt) \leq C\ll Q$.
In this case, we will evaluate the integral $\mathfrak{I}(m,n,q)$ more precisely.
The integral $\mathfrak{I}(m,n,q)$ has the following properties which will be proved
in Section \ref{proofs-of-technical-lemma}.

\begin{lemma}\label{integral:lemma-0}
Assume $V^{(j)}(x)\ll \triangle^j$ as defined in \eqref{derivative-of-V}
with $\triangle<t^{1/2-\varepsilon}$ and $C$ satisfies
$C\geq X^{1+\varepsilon}\Xi/(Qt)$.
Further assume $\varphi(x)=c\log x$ or $cx^{\beta}$ with $\beta\in (0,1), \beta\neq 1/2$.
Then we have
\bea\label{I-stationary phase-2}
\mathfrak{I}(m,n,q)
=
e\left(t\varphi(y_0^2)-Dy_0\right)\mathfrak{I}^*(m,n,q) + O_{A}(t^{-A}),
\eea
where $y_0=\left(ct/D\right)^{1/\beta}$ with $D=2q^{-1}(mX)^{1/2}$ and
\bea\label{I*}
\mathfrak{I}^*(m,n,q)
=\frac{1}{\sqrt{t}}G_{\natural}(y_*)
e\left(By_0+\frac{y_0^2}{2c\beta^2}\frac{B^2}{t}+B\sum_{j=2}^{K_2}g_{c,\beta,j}\left(y_0\right)
\left(\frac{B}{t}\right)^j\right) + O_{A}(t^{-A}).
\eea
Here $B=2q^{-1}(nX)^{1/2}$, $y_*$ is defined in \eqref{stationary point},
$ G_{\natural}(x)$ is some inert function
supported on $x\asymp 1$
and $g_{c,\beta, j}(x)$ is some polynomial function
depending only on $c,\beta,j$.
\end{lemma}

\subsubsection{Cauchy-Schwarz and Poisson summation}
Applying the Cauchy-Schwarz inequality to \eqref{beforeCauchy-2} and
using the Rankin-Selberg estimate
\eqref{GL2: Rankin Selberg}, one sees that
\bea\label{Cauchy}
\mathbf{T}&\ll&\frac{X^{1/2}}{M^{1/4}}\sum_{q\sim C}
q^{-1}\sum_{d|q}d
\bigg(\sum_{m\sim M }
|\lambda_g(m)|^2\bigg)^{1/2}\left(\sum_{m\sim M}\bigg|
\sum_{n\sim N_1\atop n\equiv m\bmod d}\lambda_f(n)n^{-1/4}\mathfrak{I}^*(m,n,q)\bigg|^2\right)^{1/2}\nonumber\\
&\ll&X^{1/2}M^{1/4}\sum_{q\sim C}q^{-1}\sum_{d|q}d\sqrt{\mathbf{\Omega}(q,d)},
\eea
where
\bea\label{Omega}
\mathbf{\Omega}(q,d)=\sum_{m\in \mathbb{Z}}\omega\left(\frac{m}{M}\right)\bigg|
\sum_{n\sim N_1\atop n\equiv m\bmod d}\lambda_f(n)n^{-1/4}\mathfrak{I}^*(m,n,q)\bigg|^2.
\eea
Here $\omega$ is a nonnegative smooth function on $(0,+\infty)$, supported on $[2/3,3]$, and such
that $\omega(x)=1$ for $x\in [1,2]$.

Opening the absolute square, we break the $m$-sum into congruence classes
modulo $d$ and apply the Poisson summation formula to the sum over $m$ to get
\bea\label{omega-bound}
\mathbf{\Omega}(q,d)
&=&\sum_{n_1\sim N_1}\lambda_f(n_1)n_1^{-1/4}
\sum_{n_2\sim N_1\atop n_2\equiv n_1\bmod d}\overline{\lambda_f(n_2)}n_2^{-1/4}
\sum_{m\equiv n_1\bmod d}\omega\left(\frac{m}{M}\right)\mathfrak{I}^*(m,n_1,q)
\overline{\mathfrak{I}^*(m,n_2,q)}\nonumber\\
&=&\frac{M}{d}\sum_{n_1\sim N_1}\lambda_f(n_1)n_1^{-1/4}
\sum_{n_2\sim N_1\atop n_2\equiv n_1\bmod d}\overline{\lambda_f(n_2)}n_2^{-1/4}
\sum_{\widetilde{m}\in \mathbb{Z}}e\left(-\frac{\widetilde{m}n_1}{d}\right)
\mathcal{H}\left(\frac{\widetilde{m}M}{d}\right),
\eea
where the integral $\mathcal{H}(x)=\mathcal{H}(x;n_1,n_2,q)$ is given by
\bea\label{H-integral}
\mathcal{H}(x)=\int_{\mathbb{R}}
\omega\left(\xi\right)
\mathfrak{I}^*\left(M\xi,n_1,q\right)
\overline{\mathfrak{I}^*\left(M\xi,n_2,q\right)}
\, e\left(-x\xi\right)\mathrm{d}\xi.
\eea

We have the following estimates for $\mathcal{H}(x)$, whose proofs we postpone to
Section \ref{proofs-of-technical-lemma}.

\begin{lemma}\label{integral:lemma} Assume
$\varphi(x)=c\log x$ or $cx^{\beta}$ with $\beta\in (0,1)\backslash \{1/2,3/4\}$.
Further assume $V^{(j)}(x)\ll \triangle^j$ as defined in \eqref{derivative-of-V}
with $\triangle<t^{1/2-\varepsilon}$ and $C$ satisfies
$C\geq X^{1+\varepsilon}\Xi/(Qt)$.

(1) We have $\mathcal{H}(x)\ll t^{-1}$ for any $x\in \mathbb{R}$.

(2) For $x\gg X^{1+\varepsilon}\Xi/(CQ)$, we have $\mathcal{H}(x)\ll_A X^{-A}$.

(3) For $x\neq 0$, we have
$\mathcal{H}(x)\ll t^{-1}|x|^{-1/2}$.

(4) $\mathcal{H}(0)$
is negligibly small unless $|n_1-n_2|\ll X^{\varepsilon}$.

\end{lemma}

With estimates for $\mathcal{H}(x)$ ready, we now continue with the treatment
of $\mathbf{\Omega}(q,d)$
in \eqref{omega-bound}. by Lemma \ref{integral:lemma} (2), the contribution from the terms with
\bea\label{m range-final}
|\widetilde{m}|\gg dC^{-1}Q^{-1}M^{-1}X^{1+\varepsilon}\Xi:=N_2,
\eea
to $\mathbf{\Omega}(q,d)$ is negligible. So
we only need to consider the range $0\leq |\widetilde{m}|\leq N_2$.

We treat the cases where $\widetilde{m}=0$ and $\tilde{m}\neq 0$ separately
and denote their contributions to $\mathbf{\Omega}(q,d)$ by
$\mathbf{\Omega}_0$ and $\mathbf{\Omega}_{\neq 0}$, respectively.

\subsubsection{The zero frequency}\label{The zero frequency}

Let $\mathbf{\Sigma}_{0}$ denote the contribution
of $\mathbf{\Omega}_0$ to \eqref{omega-bound}. Correspondingly, we denote its contribution to
\eqref{Cauchy} by $\mathbf{\Sigma}_0$.

\begin{lemma}\label{lemma:zero}
Assume
\bea\label{assumption: range 1}
Q>(X/t)^{1/2}.
\eea
We have
\bna
\mathbf{\Sigma}_0
\ll X^{\varepsilon}Q^{3/2}t.
\ena
\end{lemma}

\begin{proof}
Splitting the sum over $n_1$ and $n_2$ according as $n_1=n_2$ or not,
and applying Lemma \ref{integral:lemma} (4),
the Rankin-Selberg estimate \eqref{GL2: Rankin Selberg}
and using the inequality
$|\lambda_f(n_1)\lambda_f(n_2)|\leq |\lambda_f(n_1)|^2+|\lambda_f(n_2)|^2$, we have
\bna
\mathbf{\Omega}_0
&\ll&
\frac{M}{dtN_1^{1/2}}
\mathop{\sum\sum}_{n_1,n_2\sim N_1\atop |n_1-n_2|\ll X^{\varepsilon}}
|\lambda_f(n_1)||\lambda_f(n_2)|
\\
&\ll&\frac{M}{dtN_1^{1/2}}\sum_{n_1\sim N_1}|\lambda_f(n_1)|^2
\sum_{n_2\sim N_1\atop |n_1-n_2|\ll X^{\varepsilon}}1\\
&\ll&\frac{X^{\varepsilon}MN_1^{1/2}}{dt}.
\ena
This bound when substituted in place of $\mathbf{\Omega}(q,d)$ into \eqref{Cauchy} yields that
\bea\label{estimate-1}
\mathbf{\Sigma}_0&\ll& X^{1/2+\varepsilon}M^{1/4}\sum_{q\sim C}q^{-1}\sum_{d|q}d
\frac{M^{1/2}N_1^{1/4}}{d^{1/2}t^{1/2}}
\ll \frac{X^{1/2+\varepsilon}M^{3/4}N_1^{1/4}C^{1/2}}{t^{1/2}}.
\eea
Recall $C\ll Q$ and from \eqref{M-N1-range} that
$1\ll M\ll X^{\varepsilon}\max\left\{C^2t^2/X,X\Xi^2/Q^2\right\}$
,$N_1\asymp X\Xi^2/Q^2$.
In particular, if we further assume $Q$ satisfies $Q>(X/t)^{1/2}$, then
$1\ll M\ll X^{\varepsilon}Q^2t^2/X$. Thus
\bna
\mathbf{\Sigma}_0
\ll X^{\varepsilon}Q^{3/2}t.
\ena
This proves the lemma.

\end{proof}

\subsubsection{The non-zero frequencies}\label{The non-zero frequencies}

Recall $\mathbf{\Omega}_{\neq 0}$ denotes the contribution from the terms with $\tilde{m}\neq 0$
to $\mathbf{\Omega}(q,d)$ in
\eqref{omega-bound}. Correspondingly, we denote its contribution to
\eqref{Cauchy} by $\mathbf{\Sigma}_{\neq 0}$. Using the inequality
$|\lambda_f(n_1)\lambda_f(n_2)|\leq |\lambda_f(n_1)|^2+|\lambda_f(n_2)|^2$, we have
\bea\label{bound-med}
\mathbf{\Omega}_{\neq 0}&\ll&
\frac{M}{dN_1^{1/2}}\sum_{n_1\sim N_1}|\lambda_f(n_1)|^2
\sum_{n_2\sim N_1\atop n_2\equiv n_1\bmod d}\;
\sum_{0\neq \widetilde{m}\ll N_2}\bigg|\mathcal{H}\left(\frac{\widetilde{m}M}{d}\right)\bigg|,
\eea
where $N_2=dC^{-1}Q^{-1}M^{-1}X^{1+\varepsilon}\Xi$ is defined in \eqref{m range-final}.

\begin{lemma}\label{lemma:nonzero}
Assume
\bea\label{assumption: range 2}
Q<\min\{t,X^{1/3}\}.
\eea
We have
\bna
\mathbf{\Sigma}_{\neq 0}\ll X^{5/4+\varepsilon}/Q.
\ena
\end{lemma}
\begin{proof}
For $x=\widetilde{m}M/d$,
in order to apply the estimates for $\mathcal{H}(x)$ in Lemma \ref{integral:lemma},
we consider the cases where
$x\ll X^{\varepsilon}$ and
$x \gg X^{\varepsilon}$,
separately, and split the sum over $\tilde{m}$ accordingly.
Set
\bea\label{N3}
N_3:=dX^{\varepsilon}/M.
\eea
Then for
$0\neq \tilde{m}\ll N_3$ we will use the bound
$\mathcal{H}(x)\ll t^{-1}$ in Lemma \ref{integral:lemma}
(1), and for the remaining part we apply the bound
$\mathcal{H}(x)\ll t^{-1}|x|^{-1/2}$ in Lemma \ref{integral:lemma}
(3). By \eqref{bound-med}, we have
\bna
\mathbf{\Omega}_{\neq 0}&\ll&\frac{M}{dN_1^{1/2}}\sum_{n_1\sim N_1}|\lambda_f(n_1)|^2
\sum_{n_2\sim N_1\atop n_2\equiv n_1\bmod d}\;
\sum_{0\neq \widetilde{m}\ll N_3}t^{-1}\\
&&+\frac{M}{dN_1^{1/2}}\sum_{n_1\sim N_1}|\lambda_f(n_1)|^2
\sum_{n_2\sim N_1\atop n_2\equiv n_1\bmod d}\;
\sum_{N_3\ll \widetilde{m}\ll N_2}t^{-1}\bigg(|\widetilde{m}|M/d\bigg)^{-1/2}\\
  &\ll&\frac{MN_1^{1/2}N_3}{dt}\left(1+\frac{N_1}{d}\right)
+\frac{M^{1/2}N_1^{1/2}N_2^{1/2}}{d^{1/2}t}\left(1+\frac{N_1}{d}\right)\nonumber\\
    &\ll&\frac{M^{1/2}N_1^{1/2}}{d^{1/2}t}\left(1+\frac{N_1}{d}\right)\left(\frac{M^{1/2}N_3}{d^{1/2}}+N_2^{1/2}\right).
\ena
Here we have applied the Rankin--Selberg estimate \eqref{GL2: Rankin Selberg}.
Recall \eqref{m range-final} $N_2=dC^{-1}Q^{-1}M^{-1}X^{1+\varepsilon}\Xi$
and \eqref{N3} $N_3=dX^{\varepsilon}/M$. Thus
\bna
\mathbf{\Omega}_{\neq 0}
&\ll&\frac{X^{\varepsilon}M^{1/2}N_1^{1/2}}{t}\left(1+\frac{N_1}{d}\right)
\left(\frac{X^{\varepsilon}}{M^{1/2}}+\frac{X^{1/2}\Xi^{1/2}}{C^{1/2}Q^{1/2}M^{1/2}}\right)\\
&\ll&\frac{X^{1/2+\varepsilon}N_1^{1/2}}{C^{1/2}Q^{1/2}t}\left(1+\frac{N_1}{d}\right)
\ena
since $\Xi\ll X^{\varepsilon}$ and $Q<X^{1/2-\varepsilon}$ by \eqref{assumption 1}.

Since $1\ll M\ll X^{\varepsilon}\max\left\{C^2t^2/X,X\Xi^2/Q^2\right\}= X^{\varepsilon}C^2t^2/X$
in \eqref{M-N1-range} as $X^{1+\varepsilon}\Xi/(Qt) \leq C\ll Q$,
this bound when substituted in place of $\mathbf{\Omega}(q,d)$ in \eqref{Cauchy}
gives that
\bna
\mathbf{\Sigma}_{\neq 0}
&\ll&X^{1/2+\varepsilon}M^{1/4}\sum_{q\sim C}q^{-1}\sum_{d|q}
\frac{dX^{1/4+\varepsilon}N_1^{1/4}}{C^{1/4}Q^{1/4}t^{1/2}}\left(1+\frac{N_1^{1/2}}{d^{1/2}}\right)\\
&\ll&
\frac{M^{1/4}N_1^{1/4}X^{3/4+\varepsilon}C^{1/4}}{Q^{1/4}t^{1/2}}\left(C^{1/2}+N_1^{1/2}\right)\\
&\ll&N_1^{1/4}X^{1/2+\varepsilon}C^{3/4}Q^{-1/4}\left(C^{1/2}+N_1^{1/2}\right)\\
&\ll&N_1^{1/4}X^{1/2+\varepsilon}Q^{1/2}\left(Q^{1/2}+N_1^{1/2}\right)
\ena
Recall $1\ll N_1\asymp X\Xi^2/Q^2\ll X^{1+\varepsilon}/Q^2$ in \eqref{M-N1-range} and $C\ll Q$, we further imply
\bna
\mathbf{\Sigma}_{\neq 0}
  &\ll& X^{3/4+\varepsilon}
\left(Q^{1/2}+\frac{X^{1/2}}{Q}\right)\\
&\ll&X^{5/4+\varepsilon}/Q
\ena
provided that $Q<X^{1/3}$.

\end{proof}

\subsection{Conclusion}
By inserting the upper bounds in Lemmas \ref{lemma:zero}
and \ref{lemma:nonzero} into \eqref{Cauchy}, we have shown the following
\bna
\mathbf{T}\ll X^{\varepsilon}
\left(Q^{3/2}t+\frac{X^{5/4}}{Q}\right),
\ena
under the assumption $X^{1+\varepsilon}\Xi/(Qt) \leq C\ll Q$ and
\bea
\label{final-assumption-Q}
(X/t)^{1/2}<Q<\min\{t,X^{1/3}\}
\eea
which is a combination of \eqref{assumption 1}, \eqref{assumption: range 1} and
\eqref{assumption: range 2}. We set $Q=X^{1/2}/t^{2/5}$ to balance the contribution. Then
for $X^{1+\varepsilon}\Xi/(Qt) \leq C\ll Q$,
\bea\label{large T estimate}
\mathbf{T}\ll t^{2/5}X^{3/4+\varepsilon}
\eea
provided $X<t^{12/5}$. Moreover, for this choice of $Q$, when $C\leq X^{1+\varepsilon}\Xi/(Qt)$,
by \eqref{small contribution}, $\mathbf{T}$ is bounded by
\bna
\frac{X^{2+\vartheta+\varepsilon}}{Q^{3+2\vartheta}t^{1/2}}=X^{1/2+\varepsilon}t^{7/10+4\vartheta/5}.
\ena
Note that the estimate $t^{2/5}X^{3/4+\varepsilon}$ is superior to the trivial bound
$X^{1+\varepsilon}$ for $X>t^{8/5}$. So we assume $X>t^{8/5}$ and in this case the term
$X^{1/2+\varepsilon}t^{7/10+4\vartheta/5}$ is dominated by
the estimate in \eqref{large T estimate}, since
we can take $\vartheta=7/64$ by \cite{KS}.

Substituting the estimate in \eqref{large T estimate}
for $\mathbf{T}$ into \eqref{M-N1-range} and using \eqref{C range}, we obtain
\bna
 S(X,t)\ll t^{2/5}X^{3/4+\varepsilon}
\ena
provided $t^{8/5}<X<t^{12/5}$.

Notice that by Lemma \ref{integral:lemma} (3), $\varphi(x)$
further satisfies $\varphi(x)=c\log x$ or $cx^{\beta}$
($\beta\in (0,1)\backslash \{1/2,3/4\}$, $c\in \mathbb{R}\backslash \{0\}$); see \eqref{phi assumption}. The assumption $\triangle<t^{1/2-\varepsilon}$ arises also in applying Lemma \ref{integral:lemma} (3); see \eqref{assumption-on-Delta}.
This completes the proof of Theorem \ref{main-theorem}.

\section{Proof of Corollary  \ref{sharp-cut-sum}}\label{proofs-of-corollary}
In this section, we prove Corollary \ref{sharp-cut-sum} in Section \ref{introduction}.
Without loss of generality, we assume that $f$ and $g$ are both Maass cusp forms.
If one of $f$ and $g$ is holomorphic, the proof is similar and simpler.
Note that from \eqref{GL2: Rankin Selberg}, we have
\bna
\sum_{X<n\leq X+X/\triangle}|\lambda_f(n)|^2\ll X/\triangle+X^{3/5}.
\ena
In particular, if $\triangle\leq X^{2/5}$, one has
\bna
\sum_{X<n\leq X+X/\triangle}|\lambda_f(n)|^2\ll X/\triangle.
\ena
Similarly, under the same assumption $\triangle\leq X^{2/5}$, we have
\bna
\sum_{X<n\leq X+X/\triangle}|\lambda_g(n)|^2\ll X/\triangle.
\ena

We choose the smooth function $V$ in \eqref{natural-sum} to be supported on
$[1,2]$ and $V(x)=1$ on $[1+1/\triangle,2-1/\triangle]$. Then, Theorem \ref{main-theorem} yields
\bna
&&\sum_{X<n\leq 2X}\lambda_f(n)
\lambda_g(n)e\left(t \varphi\left(\frac{n}{X}\right)\right)\\
&\ll&
t^{7/16}X^{3/4+\varepsilon}+
\sum_{X<n\leq X+X/\triangle}|\lambda_f(n)\lambda_g(n)|+
\sum_{2X-X/\triangle<n\leq 2X}|\lambda_f(n)\lambda_g(n)|\\
&\ll& t^{2/5}X^{3/4+\varepsilon}+
\bigg(\sum_{X<n\leq X+X/\triangle}|\lambda_f(n)|^2\bigg)^{1/2}
\bigg(\sum_{X<n\leq X+X\triangle}|\lambda_g(n)|^2\bigg)^{1/2}\\
&\ll& t^{2/5}X^{3/4+\varepsilon}+X/\triangle,
\ena
as long as $\triangle\leq X^{2/5}$ and $t^{8/5}<X<t^{12/5}$.
Corollary \ref{sharp-cut-sum}
then follows by choosing $\triangle=t^{1/2-\varepsilon}$ and
by noting that $t^{1/2-\varepsilon}\leq X^{2/5}$
if and only if $t^{5/4-\varepsilon}\leq X$.

\section{Proof of Corollary \ref{cor:main}} \label{sec:Cor}
In this section, we prove Corollary \ref{cor:main}. We introduce a
lemma about "Functional equation" of $L(s,1\boxplus(f \times g))$ before
proving Corollary \ref{cor:main}.

\begin{lemma}\label{lemma:functional equation}
For $\rm {Re}(s) > 1$, we have
$$
L(1-s,1\boxplus(f \times g))= \frac{1}{\varepsilon(f\times g)}\gamma(s)L(s,1\boxplus(f \times g)),
$$
where $\varepsilon(f\times g)$ is the root number of
$L(f \times g)$ with $|\varepsilon(f\times g)| = 1$,
$$\gamma(s)={(\pi^{-5})}^{s-\frac{1}{2}}\prod_{j=1}^{5} \Gamma(\frac{s+\kappa_{j}}{2})\Gamma(\frac{1-s+\kappa_{j}}{2})^{-1},$$
with $\kappa_{1}=0, \kappa_{2}=\frac{k-\kappa}{2}, \kappa_{3}=\frac{k-\kappa}{2}+1, \kappa_{4}=\frac{k+\kappa}{2}-1, \kappa_{5}=\frac{k+\kappa}{2}$,
and
$$
\gamma(\sigma- i t) = \overline{\omega_k} \left(\frac{t}{2 \pi} \right)^{5(\sigma-1 / 2)} \left(\frac{2 \pi e}{t}\right)^{5 i t}\left\{1+O\left(\frac{1}{t}\right)\right\},
$$
for $\sigma>1 / 2$, $t>1$,
$\omega_k=e\left(\frac{4 k-5}{8}\right)$.
\end{lemma}

\begin{proof}
First by the functional equation
  \begin{equation*}
  \begin{split}
    \Lambda(s, 1\boxplus(f \times g))& = \pi^{-\frac{s}{2}} \Gamma\left(\frac{s}{2}\right) \zeta(s) \cdot (2 \pi)^{-2s}\Gamma\left(s+\frac{k-\kappa}{2}\right)\Gamma\left(s+\frac{k+\kappa}{2}-1 \right) L(s, f \times g)\\
    & = \varepsilon(1\boxplus(f\times g))\Lambda(1-s, 1\boxplus(f \times g)) = \varepsilon(f\times g)\Lambda(1-s, 1\boxplus(f \times g)) ,
  \end{split}
  \end{equation*}
we can write the functional equation as follows,
$$
\begin{aligned}
&L(1-s, 1 \boxplus f \times g) \\
=&\frac{1}{\varepsilon(f\times g)} \frac{\pi^{-\frac{s}{2}} \Gamma\left(\frac{s}{2}\right)(2 \pi)^{-2 s} \Gamma\left(s+\frac{k-\kappa}{2}\right)\Gamma\left(s+\frac{k+\kappa}{2}-1 \right)}{\pi^{-\frac{1-s}{2}} \Gamma\left(\frac{1-s}{2}\right)(2 \pi)^{-2+2s}\Gamma\left(1-s+\frac{k-\kappa}{2}\right)\Gamma\left(1-s+\frac{k+\kappa}{2}-1 \right)} L(s, 1 \boxplus f \times g) \\
=&\frac{1}{\varepsilon(f\times g)} \gamma(s) L(s, 1 \boxplus (f \times g)),
\end{aligned}
$$
where
$$
\gamma(s)=\pi^{\frac{1}{2}-s}(2 \pi)^{2-4 s}\frac{\Gamma\left(\frac{s}{2}\right)}{\Gamma\left(\frac{1-s}{2}\right)} \frac{\Gamma\left(s+\frac{k-\kappa}{2}\right)}{\Gamma\left(1-s+\frac{k-\kappa}{2}\right)} \frac{\Gamma\left(s+\frac{k+\kappa}{2}-1 \right)}{\Gamma\left(1-s+\frac{k+\kappa}{2}-1 \right)}.
$$
Note that $\Gamma(z) \Gamma\left(z+\frac{1}{2}\right)=2^{1-2 z} \pi^{\frac{1}{2}} \Gamma(2 z)$. Taking $z=\frac{s+\frac{k-\kappa}{2}}{2}$, $z=\frac{s+\frac{k+\kappa}{2}-1}{2}$ respectively, we obtain
$$
\Gamma\left(s+\frac{k-\kappa}{2}\right)=2^{s+\frac{k-\kappa}{2}-1} \pi^{-\frac{1}{2}} \Gamma\left(\frac{s+\frac{k-\kappa}{2}}{2}\right)\Gamma\left(\frac{s+\frac{k-\kappa}{2}+1}{2}\right),
$$
$$
\Gamma\left(s+\frac{k+\kappa}{2}-1\right)=2^{s+\frac{k+\kappa}{2}-2} \pi^{-\frac{1}{2}} \Gamma\left(\frac{s+\frac{k+\kappa}{2}-1}{2}\right)\Gamma\left(\frac{s+\frac{k+\kappa}{2}}{2}\right).
$$
With the notation of $\gamma(s)$ , we derive
$$
\gamma(s)
=\left(\pi^{-5}\right)^{s-\frac{1}{2}} \prod_{j=1}^{5} \Gamma\left(\frac{s+\kappa_{j}}{2}\right) \Gamma\left(\frac{1-s+\kappa_{j}}{2}\right)^{-1},
$$
with $\kappa_{1}=0, \kappa_{2}=\frac{k-\kappa}{2},
\kappa_{3}=\frac{k-\kappa}{2}+1, \kappa_{4}=\frac{k+\kappa}{2}-1,
\kappa_{5}=\frac{k+\kappa}{2}$. Hence,
by the argument of Friedlander-Iwaniec \cite[Section 1]{Fri-Iwa}, we complete the proof of Lemma.
\end{proof}
Take $s=1+\varepsilon-i t$. Lemma \ref{lemma:functional equation} yields
\begin{equation}\label{apply lemma}
  L(-\varepsilon+i t, 1 \boxplus (f \times g))=
  \frac{\overline{\omega_k} }{\varepsilon(f\times g)}
  \left(\frac{t}{2\pi}\right)^{\frac{5}{2}+5\varepsilon}\left(\frac{t}{2\pi e}\right)^{-5 i t} L(1+\varepsilon-i t,1 \boxplus (f \times g))\left\{1+O\left(\frac{1}{t}\right)\right\}.
\end{equation}

\begin{proof}[Proof of Corollary \ref{cor:main}]

We first approximate $\sum_{n \leq X}\lambda_{1\boxplus(f\times g)}(n)$ by a smooth sum. Let
$$
Y=X^{2/3-\delta} \quad \text { for some } \delta \in (0, 2/39).
$$
Let $W$ be a smooth function supported on $[1/2 - Y/X, 1 + Y/X]$
such that $W(u)=1$, $u \in [1/2, 1]$ and $W(u) \in [0 , 1]$,
$u \in [1/2 - Y/X, 1/2] \cup [1,1+Y / X]$, and $W^{(m)}(u) \ll_ m(X / Y)^{m}$ for all $m \geq 1$.
Then
\begin{equation}\label{sum of lambda 1}
\begin{aligned}
\sum_{X / 2<n \leq X} \lambda_{1\boxplus(f\times g)}(n)=& \sum_{X / 2-Y<n<X+Y} \lambda_{1\boxplus(f\times g)}(n) W\left(\frac{n}{X}\right) \\
&+O\left(\sum_{X / 2-Y<n<X / 2}\left|\lambda_{1\boxplus(f\times g)}(n)\right|+\sum_{X<n<X+Y}\left|\lambda_{1\boxplus(f\times g)}(n)\right|\right) \\
=& \sum_{n \geq 1} \lambda_{1\boxplus(f\times g)}(n) W\left(\frac{n}{X}\right)+O\left(X^{2/3-\delta+\varepsilon}\right),
\end{aligned}
\end{equation}
where we have used Deligne's bound
$\lambda_{1\boxplus(f\times g)}(n) =
\sum_{lm^2r=n}\lambda_f(r)\lambda_g(r)\ll n^{\varepsilon}$.
Thus we only need to show
\begin{equation}\label{sum of lambda 2}
  \sum_{n \geq 1} \lambda_{1\boxplus(f\times g)}(n) W\left(\frac{n}{X}\right)=L(1, f\times g) \tilde{W}(1) X + O\left(X^{2/3-\delta + \varepsilon}\right),
\end{equation}
where $\tilde{W}(s)=\int_{0}^{\infty} W(x) x^{s-1} \mathrm{d} x$ is
the Mellin transform of $W(x)$ and $\tilde{W}(1)= 1/2 + O(Y/X)$.
By breaking the sum into dyadic intervals and plugging \eqref{sum of lambda 2}
into \eqref{sum of lambda 1}, we get
$$
\begin{aligned}
\sum_{n \leq X} \lambda_{1\boxplus(f\times g)}(n) =& 2 L(1, f\times g) \tilde{W}(1) X + O\left(X^{2/3 - \delta + \varepsilon}\right)\\
=&L(1, f\times g)X + O\left(X^{2/3 - \delta + \varepsilon}\right).
\end{aligned}
$$
Now we consider the sum $\sum_{n \geq 1}\lambda_{1\boxplus(f\times g)}(n)W\left(\frac{n}{X}\right)$ in \eqref{sum of lambda 2}. By the Mellin inversion formula
$$
W(u)=\frac{1}{2 \pi i} \int_{(2)} \tilde{W}(s) u^{-s} \mathrm{d} s,
$$
 we get
$$
\sum_{n \geq 1}\lambda_{1\boxplus(f\times g)}(n)W\left(\frac{n}{X}\right)= \frac{1}{2 \pi i} \int_{(2)} \tilde{W}(s) L(s, 1 \boxplus (f \times g)) X^{s} \mathrm{d}s.
$$
Next we move the integration to the parallel segment with $\mathrm{Re}( s ) = -\varepsilon$.
Note that inside the contour the integrand has only a
simple pole at $s = 1$ with residue $L(1,f \times g)\tilde{W}(1)X$, since $L(s,1\boxplus(f \times g))= \zeta(s)L(s,f \times g)$.
Hence,
\begin{equation}\label{move integral}
  \sum_{n \geq 1} \lambda_{1\boxplus(f\times g)}(n) W\left(\frac{n}{X}\right)
  =L(1, f\times g) \tilde{W}(1) X + \frac{1}{2 \pi i}\int_{(-\varepsilon)}\tilde{W}(s)
  L(s, 1 \boxplus (f \times g)) X^{s} \mathrm{d}s.
\end{equation}
Let
$$
I(X):= \frac{1}{2 \pi i}\int_{(-\varepsilon)}\tilde{W}(s)
L(s, 1 \boxplus (f \times g)) X^{s} \mathrm{d}s.
$$
Inserting a dyadic smooth partition of unity to the $t$-integral, we get
\begin{equation}\label{I(X)dyadic}
 I(X)=\sum_{T \text { dyadic }} I(X, T),
\end{equation}
where
$$
I(X, T):=\frac{X^{-\varepsilon}}{2 \pi} \int_{\mathbb{R}} X^{i t} \tilde{W}(-\varepsilon+i t) L(-\varepsilon+i t, 1 \boxplus (f \times g)) V\left(\frac{t}{T}\right) \mathrm{d} t
$$
for some fixed compactly supported function $V$. For $\tilde{W}(s)$,
by applying integration by parts, we have, for any $m \geq 1$
\begin{equation}\label{tilde{W}(s)}
 \tilde{W}(s)=\frac{(-1)^{m}}{s(s+1) \cdots(s+m-1)} \int_{0}^{\infty} W^{(m)}(u) u^{s+m-1} \mathrm{d} u \ll_m \frac{1}{|s|^{m}}\left(\frac{X}{Y}\right)^{m-1},
\end{equation}
since supp $W^{(m)}\subset [1/2 - Y/X, 1/2] \cup [1,1 + Y /X]$.
By \eqref{tilde{W}(s)}, one finds that the contribution
from the $t$-integral of $I(X, T )$ is negligible if
$t \gg X^{1+\varepsilon}/Y$. In addition, by
the upper bound $L(-\varepsilon+i t, 1 \boxplus (f \times g))
\ll (1+ t)^{5/2+\varepsilon}$ which follows from Lemma \ref{lemma:functional equation} and the Phragm\'{e}n \textendash Lindel\"{o}f principle and by \eqref{tilde{W}(s)} with $m=1$, we deduce that
$$
I(X,T)\ll X^{\varepsilon}T^{5/2 +\varepsilon}\ll Y
$$
if $T\ll Y^{2/5-\varepsilon}$.
Thus, up to a negligible error, we only need to consider
those $T$ in \eqref{I(X)dyadic} in the range $ Y^{2/5-\varepsilon}\ll T \ll X^{1+\varepsilon}/Y$.
And we only consider positive $T$'s, since negative $T$'s can be handled
similarly. Next, for $I(X, T )$, by the first equality in \eqref{tilde{W}(s)} with $m = 1$, we get
\begin{equation}\label{I(X,T)}
\begin{aligned}
I(X, T) &=-\frac{X^{-\varepsilon}}{2 \pi} \int_{1 / 3}^{3} W^{\prime}(u) u^{-\varepsilon} \int_{\mathbb{R}} \frac{(X u)^{i t}}{-\varepsilon+i t} L(-\varepsilon+i t, 1\boxplus(f\times g)) V\left(\frac{t}{T}\right) \mathrm{d} t~\mathrm{d} u \\
& \ll \frac{X^{-\varepsilon}}{T} \sup _{u \in [1/3,3]} \left| \int_{\mathbb{R}}(X u)^{i t} L(-\varepsilon+i t, 1\boxplus(f\times g)) V_1\left(\frac{t}{T}\right) \mathrm{d} t \right|.
\end{aligned}
\end{equation}
Hence, in the following, we only need to estimate
\begin{equation}\label{J(X, T)}
 J(X, T):=\int_{\mathbb{R}} X^{i t} L(-\varepsilon+i t, 1\boxplus(f\times g)) V_1\left(\frac{t}{T}\right) \mathrm{d} t.
\end{equation}
We shall apply functional equation for $L(-\varepsilon+i t, 1\boxplus(f\times g))$
to change the variable $s=-\varepsilon+i t$ into $1-s=1+\varepsilon-i t$.
By inserting the functional equation \eqref{apply lemma} into \eqref{J(X, T)}, we have
$$
\begin{aligned}
J(X, T)=& \int_{\mathbb{R}} X^{i t} \frac{1}{\varepsilon(f\times g)} \overline{\omega_{k}} \cdot\left(\frac{t}{2 \pi}\right)^{5\left(\frac{1}{2}+\varepsilon\right)}\left(\frac{t}{2 \pi e}\right)^{-5 i t} L(1+\varepsilon-i t,1\boxplus(f\times g)) V_1\left(\frac{t}{T}\right) \mathrm{d} t \\
&+O\left(\frac{1}{T} \cdot T^{5 / 2+\varepsilon} \cdot T\right)\\
\ll & T^{5 / 2+\varepsilon}\left|\int_{\mathbb{R}}\sum_{n\geq 1}\frac{\lambda_{1\boxplus(f\times g)}(n)}{n^{1+\varepsilon-i t}}X^{i t}\left(\frac{t}{2 \pi e}\right)^{-5 i t} V_2\left(\frac{t}{T}\right)\mathrm{d} t\right| +T^{5 / 2+\varepsilon}
\end{aligned}
$$
for some smooth compactly supported function $V_2$.

Exchanging the order of the integration and summation above,
and making a change of variable $\frac{t}{T}\rightarrow \xi$, we get
$$
\begin{aligned}
J(X,T)\ll & T^{5 / 2+\varepsilon}\left|\sum_{n\geq 1}
\frac{\lambda_{1\boxplus(f\times g)}(n)}{n^{1+\varepsilon}}\int_{\mathbb{R}}
(nX)^{i t}\left(\frac{t}{2 \pi e}\right)^{-5 i t}
_2\left(\frac{t}{T}\right)\mathrm{d} t\right| + T^{5 / 2+\varepsilon} \\
\ll & T^{7 / 2+\varepsilon}\left|\sum_{n\geq 1}
\frac{\lambda_{1\boxplus(f\times g)}(n)}{n^{1+\varepsilon}}\int_{\mathbb{R}}
e^{i T \xi \log (nX({\frac{2 \pi e}{T\xi}})^{-5})} V_2(\xi)\mathrm{d}\xi\right|
+ T^{5 / 2+\varepsilon}.
\end{aligned}
$$
Let $h(\xi):= T \xi \log (nX({\frac{2 \pi e}{T\xi}})^{-5})$,
then $h^{\prime}(\xi)= 5T\log\frac{2\pi (nX)^{\frac{1}{5}}/T}{\xi}$, $h^{(j)}(\xi)= (-1)^{j-1}(j-2)!\frac{5T}{\xi^{j-1}}$ for $j \geq 2$.
If $2\pi(nX)^{1/5}/T \notin \supp V_{2} $, it is not difficult
to see that $h^{\prime}(\xi) \gg T^{\varepsilon}$.
Applying Lemma 3.6 (1), we have the integral over $\xi$ is
$O(T^{-2021})$.
Now for the above integral over $\xi$, we consider the case $2\pi(nX)^{1/5}/T \in \supp V_2$.
Note that the stationary point is $\xi_{0}=\frac{2 \pi(n X)^{1 / 5}}{T}$, $h\left(\xi_{0}\right)=5 T \xi_{0}$, $h^{\prime \prime}\left(\xi_{0}\right)=-\frac{5 T}{\xi_{0}} \asymp T$ and $V^{(j)}(\xi)\ll_j 1$ for $j \geq 0$, $h^{(j)}\left(\xi_{0}\right)\asymp T$ for $j \geq 2$.
Applying Lemma 3.6 (2)
with $Y=Z=1$ and $H=R=T$, we obtain
$$
\begin{aligned}
\int_{\mathbb{R}} V_{2}(\xi) e^{i T \xi \log \left(n X\left(\frac{T \xi}{2 \pi e}\right)^{-5}\right)} \mathrm{d} \xi &=\frac{e^{i h\left(\xi_{0}\right)}}{T^{1 / 2}} W_{1}\left(\xi_{0}\right)+O\left(\frac{1}{T^{2021}}\right) \\
&=\frac{e\left(5(n X)^{1 / 5}\right)}{T^{1 / 2}} W_{2}\left(\frac{n}{T^{5} / X}\right)+O\left(\frac{1}{T^{2021}}\right),
\end{aligned}
$$
for some inert functions $W_{1}$, $W_{2}$.
Consequently,
\begin{equation}\label{after stationary}
\begin{aligned}
J(X, T) & \ll T^{3+\varepsilon}\left|\sum_{n \geq 1} \frac{\lambda_{1\boxplus(f\times g)}(n)}{n^{1+\varepsilon}} e\left(5(n X)^{1 / 5}\right) W_{2}\left(\frac{n}{T^{5} / X}\right)\right|+T^{5 / 2+\varepsilon} \\
& \ll \frac{X^{1+\varepsilon}}{T^2}\left|\sum_{n \geq 1} \lambda_{1\boxplus(f\times g)}(n) e\left(5(n X)^{1 / 5}\right) W_{3}\left(\frac{n}{T^{5} / X}\right)\right|+T^{5 / 2+\varepsilon}
\end{aligned}
\end{equation}
for some inert function $W_3$.
Note that
\begin{equation}\label{T range}
 X^{1/5+\varepsilon} \ll Y^{2/5+\varepsilon} \ll T \ll \frac{X^{1+\varepsilon}}{Y}.
\end{equation}
So the above sum over $n$ is non-empty. Combining \eqref{I(X)dyadic}, \eqref{I(X,T)}, \eqref{J(X, T)} and \eqref{after stationary}, we have
\begin{equation}\label{I(X) to exp sum}
 I(X)\ll \sum_{T ~{\rm{dyadic}} \atop Y^{2/5+\varepsilon} \ll T \ll \frac{X^{1+\varepsilon}}{Y}}\left(\frac{X^{1+\varepsilon}}{T^{3}}\left|\sum_{n \geq 1} \lambda_{1 \boxplus(f\times g)}(n) e\left(5(n X)^{1 / 5}\right) W\left(\frac{n}{T^{5} / X}\right)\right|+T^{3/2+\varepsilon}\right).
\end{equation}
Here $X$ on the right-hand side of \eqref{I(X) to exp sum} should be
understood as the original $X u$ in \eqref{I(X,T)} with $u \in [1/3 , 3]$,
and $W$ is a smooth compactly supported function with $\supp W \in [1 / 4,4]$.
So we only need to consider the case $n \asymp T^{5} / X$.
Now we make use of the fact that
$\lambda_{1\boxplus(f\times g)}(n) = \sum_{lm^2r=n}\lambda_f(r)\lambda_g(r)$.
Inserting dyadic partitions to the $l$-sum and $m$-sum and making a
smooth partition of unity into dyadic segments to the $r$-sum, we arrive at
$$
I(X)\ll \sum_{T ~{\rm{dyadic}} \atop Y^{2/5+\varepsilon} \ll T \ll \frac{X^{1+\varepsilon}}{Y}}\left(\frac{X^{1+\varepsilon}}{T^{3}}
\sup_{L,M,R \gg 1 \atop LM^2R \asymp T^5/X}|B(L,M,R)| + T^{3/2+\varepsilon}\right),
$$
where
$$
B(L,M,R):=\sum_{l \sim L} \sum_{m \sim M} \sum_{r \geq 1}\lambda_{f}(r)\lambda_{g}(r) e\left(5(l m^2 r X)^{1 / 5}\right) V\left(\frac{r}{R}\right).
$$

We distinguish two cases.

{\textbf{Case 1.}} $L \gg T^{593 / 345}M^{-194 / 207}X^{-97 / 207}$.
We rewrite $B(L,M,R)$ as
$$
B(L, M, R)=\sum_{m \sim M} \sum_{r \geq 1} \lambda_{f}(r)\lambda_{g}(r) V\left(\frac{r}{R}\right)\left(\sum_{l \sim L }  e\left(5(l m^2 r X)^{1 / 5}\right)\right).
$$
For the inner sum over $l$, we apply the method of exponent pairs with A-process
(see for example \cite[Chapter 3]{GrahamKolesnik}), by taking the exponent pair $(p, q)$ as
$$
(p, q)=\left(\frac{k}{2 k+2}, \frac{k+h+1}{2 k+2}\right)=
\left(\frac{13}{194}+\varepsilon, \frac{76}{97}+\varepsilon\right),
$$
where $(k, h)=\left(\frac{13}{84}+\varepsilon, \frac{55}{84}+\varepsilon\right)$
is an exponent pair according to Bourgain \cite[Theorem 6]{Bourgain}. Hence,
\begin{equation}\label{B(L, M, R)1}
\begin{aligned}
B(L, M, R)& \ll \sum_{m \sim M} \sum_{r \geq 1} \left|\lambda_{f}(r)\lambda_{g}(r) V\left(\frac{r}{R}\right)\right| \left|\sum_{l \sim L }  e\left(5(l m^2 r X)^{1 / 5}\right)\right|\\
& \ll T^{\varepsilon} M R \cdot(T / L)^{p} L^{q}\\
&\ll T^{983/194+\varepsilon}X^{-1+\varepsilon}M^{-1+\varepsilon}L^{-55/194+\varepsilon}\\
& \ll T^{316 /69+\varepsilon} M^{-152 / 207+\varepsilon} X^{-359 / 414+\varepsilon}\\
& \ll T^{316 /69+\varepsilon} X^{-359 / 414+\varepsilon}.
\end{aligned}
\end{equation}
In the last inequality we have used the fact $M\gg1$.

{\textbf{Case 2.}} $L \ll T^{593/ 345}M^{-194 / 207}X^{-97 / 207}$.
We rewrite $B(L, M, R)$ as
\begin{equation}\label{B(L, M, R)2}
\begin{aligned}
B(L, M, R)& =\sum_{l \sim L}\sum_{m \sim M} \left(\sum_{r \geq 1}\lambda_{f}(r)\lambda_{g}(r) e\left(5(l m^2 r X)^{1 / 5}\right) V\left(\frac{r}{R}\right) \right)\\
& \ll \sum_{l \sim L}\sum_{m \sim M} \left|\sum_{r \geq 1}\lambda_{f}(r)\lambda_{g}(r) e\left(5T(\frac{r}{R})^{1 / 5}\right) V\left(\frac{r}{R}\right) \right|.
\end{aligned}
\end{equation}
In order to apply Theorem 1.1, we need to verify that $R$ satisfies
the condition $R \ll T^{12/5}$.
Note that $R \ll LM^2R \asymp T^5/X$ and $Y^{2/5+\varepsilon} \ll T \ll \frac{X^{1+\varepsilon}}{Y}
= X^{1/3+\delta+\varepsilon}$.
Since we assume $\delta < 2/39 $, we have $T\ll X^{5/13}$ and hence $R\ll T^5/X \ll T^{12/5}$.
Therefore, by Theorem 1.1, we have
$$
\begin{aligned}
B(L, M, R) &\ll L M T^{\frac{2}{5}}R^{\frac{3}{4}+\varepsilon} \asymp L^{1/4+\varepsilon}M^{-1/2+\varepsilon}T^{83/20+\varepsilon}X^{-3/4+\varepsilon}\\
&\ll T^{316 /69+\varepsilon} M^{-152 / 207+\varepsilon} X^{-359 / 414+\varepsilon}\\
&\ll T^{316 /69+\varepsilon}X^{-359 / 414+\varepsilon}.
\end{aligned}
$$
Combining \eqref{I(X) to exp sum}, \eqref{B(L, M, R)1} and \eqref{B(L, M, R)2},
we have
\begin{equation}\label{I(X)1}
\begin{aligned}
I(X)
&\ll \sum_{T ~{\rm{dyadic}} \atop Y^{2/5+\varepsilon} \ll T \ll  X^{1 /3+\delta+\varepsilon}}\left(\frac{X^{1+\varepsilon}}{T^{3}} \cdot T^{316 /69+\varepsilon} X^{-359 / 414+\varepsilon} + T^{3 / 2+\varepsilon}\right)\\
& \ll \sum_{T ~{\rm{dyadic}} \atop Y^{2/5+\varepsilon} \ll T \ll  X^{1 /3+\delta+\varepsilon}} \left(X^{55 / 414 + \varepsilon} T^{109 /69+\varepsilon}  + T^{3 / 2 + \varepsilon}\right)  \\
& \ll X^{\frac{109}{69}\delta + \frac{91}{138} +\varepsilon}.
\end{aligned}
\end{equation}
Finally, putting together the above estimates \eqref{sum of lambda 1},
\eqref{move integral} and \eqref{I(X)1}, we conclude that
$$
\sum_{X / 2<n \leq X} \lambda_{1\boxplus(f\times g)}(n)= L(1, f\times g) \tilde{W}(1) X +
O\left(X^{\frac{109}{69}\delta + \frac{91}{138}+ \varepsilon}\right) +
O\left(X^{2/3 - \delta + \varepsilon}\right).
$$
So we complete the proof of Corollary \ref{cor:main} by taking $\delta \leq 1/356$.
\end{proof}

\section{Estimation of integrals}\label{proofs-of-technical-lemma}

We first prove Lemma  \ref{integral:lemma-0}.
\begin{proof}[Proof of Lemma \ref{integral:lemma-0}]

By \eqref{I-change-0}, we write
\bna
\mathfrak{I}(m,n,q)
=2\int_0^\infty y\widetilde{V}(y^2)
e\left(t\varphi(y^2)+By-Dy\right)\mathrm{d}y,
\ena
where
\bea\label{BD}
B=2q^{-1}(nX)^{1/2}\asymp \sqrt{XN_1}/C, \qquad
D=2q^{-1}(mX)^{1/2}\asymp (MX)^{1/2}/C.
\eea
Recall the range of $N_1$ in \eqref{M-N1-range} that $N_1\asymp X\Xi^2/Q^2$.
Thus for $X^{1+\varepsilon}\Xi/(Qt) \leq C\ll Q$, we have
\bea\label{B upper bound}
B\ll \frac{X^{1+\varepsilon}\Xi}{CQ}\ll X^{-\varepsilon}t.
\eea
Therefore, the integral $\mathfrak{I}(m,n,q)$ is negligibly small unless
$D\asymp t$.

Assume
\bea\label{phi assumption}
(\varphi(y^2))'=cy^{-\beta} \qquad \text{with}\quad \beta\neq 0,
\eea
where $c>0$ is an absolute constant, i.e.,
\bea\label{phi assumption-2}
\varphi(y)=\frac{c}{2}\log y+c_1\qquad \text{or}\qquad
\varphi(y)=\frac{c}{1-\beta}y^{(1-\beta)/2}+c_2
\quad \text{with}\; \beta\neq 0,1,
\eea
where $c_i\in \mathbb{R}, i=1,2$, are absolute constants. Without loss of generality,
we further assume
$c_i=0,i=1,2$.
Let $\rho(y)=t\varphi(y^2)+By-Dy$.
Then
\bna
\rho'(y)&=&cty^{-\beta}+B-D,\\
\rho^{(j)}(y)&=&t\big(\varphi(y^2)\big)^{(j)}, \quad j=2,3,\ldots.
\ena
The stationary point $y_*$ which is the solution to
the equation $\rho'(y)=cty^{-\beta}+B-D$ is $y_*=\left(\frac{ct}{D-B}\right)^{1/\beta}$.
Denote
\bna
C_{\alpha}^j=\frac{\alpha(\alpha-1)\cdot\cdot\cdot (\alpha-j+1)}{j!}.
\ena
Then by the Taylor series approximation, $y_*$ can be written as
\bea\label{stationary point}
y_*&=&\left(\frac{ct}{D}\right)^{1/\beta}\bigg(1+ \sum_{j=1}^{K_1}
C_{-1/\beta}^{j}\left(\frac{-B}{D}\right)^j
+O_{\beta,K_1}\left(\frac{B^{K_1+1}}{t^{K_1+1}}\right) \bigg)  \nonumber\\
&:=&y_0\left(1+\sum_{j=1}^{K_1}y_j
+O_{c,\beta,K_1}\left(\frac{B^{K_1+1}}{t^{K_1+1}}\right) \right),
\eea
where here and after, $K_j\geq 1$, $j=1,2,3\ldots,$ denote integers, and
\bna
y_0&=&\left(\frac{ct}{D}\right)^{1/\beta} \asymp 1,\\
y_{j}&=&C_{-1/\beta}^{j}\left(\frac{-B}{D}\right)^j
\asymp \left(\frac{B}{t}\right)^{j}.
\ena
By \eqref{B upper bound}, the
$O$-term in \eqref{stationary point} is $O(N^{-\varepsilon K_1 })$, which
can be arbitrarily small by taking $K_1$ sufficiently large.

Note that $\rho^{(j)}(y)\asymp t$ for any integer $j\geq 1$.
Recall $\widetilde{V}^{(j)}(y)\ll_j \triangle^j$,
where $\triangle<t^{1-\varepsilon}$ (see \eqref{derivative-of-V}).
To make sure that the stationary phase analysis is applicable to the integral
$\mathfrak{I}(M\xi,n,q)$, we assume $\triangle$ satisfies
\bea\label{assumption-on-Delta}
\triangle<t^{1/2-\varepsilon}
\eea
Now applying Lemma \ref{lemma:exponentialintegral} with $Z=1$, $Y=\triangle$,
$H=t$ and $R=H/X^2\gg t^{\varepsilon}$, we have
\bna
\mathfrak{I}(m,n,q)
=\frac{e(\rho(y_*))}{\sqrt{2\pi \rho''(y_*)}}
  G(y_*) + O_{A}(  t^{-A}),
\ena
for any $A>0$, where $G(y)$ is some inert function
supported on $y\asymp 1$. From \eqref{phi assumption-2}, \eqref{stationary point} and
using Taylor series approximation, we have
\bna
\rho(y_*)&=&t\varphi(y_*^2)+By_*-Dy_*\\
&=&t\varphi(y_0^2)-Dy_0+By_0+\frac{y_0^2}{2c\beta^2}\frac{B^2}{t}+B\sum_{j=2}^{K_2}g_{c,\beta,j}\left(y_0\right)
\left(\frac{B}{t}\right)^j+O_{c,\beta,K_2}\left(\frac{B^{K_2+2}}{t^{K_2+1}}\right)
\ena
and
\bna
\rho''(y_*)=-c\beta ty_*^{-\beta-1}
=-c\beta ty_0^{-\beta-1}+B(\beta+1)y_0^{-1}+
B\sum_{j=1}^{K_3}h_{c,\beta,j}\left(y_0\right)
\left(\frac{B}{t}\right)^j+O_{c,\beta,K_3}\left(\frac{B^{K_3+2}}{t^{K_3+1}}\right)
\ena
for some functions $g_{c,\beta, j}(x)$, $h_{c,\beta, j}(x)$ of polynomially growth,
depending only on $c,\beta,j$, and supported on $x\asymp 1$.
Note that $\rho''(y_*)\asymp t$.
Hence,
\bna
\mathfrak{I}(m,n,q)
&=&\frac{1}{\sqrt{t}}G_{\natural}(y_*)
e\left(t\varphi(y_0^2)-Dy_0+By_0+\frac{y_0^2}{2c\beta^2}\frac{B^2}{t}\right)\nonumber\\
&&\qquad\times e\left(B\sum_{j=2}^{K_2}g_{c,\beta,j}\left(y_0\right)
\left(\frac{B}{t}\right)^j\right) + O_{A}(t^{-A}),
\ena
where $ G_{\natural}(y)=\left(t/(2\pi \rho''(y))\right)^{1/2}
G(y)$ satisfies $G_{\natural}^{(j)}(y)\ll_j 1$.
This finishes the proof of the lemma.

\end{proof}

Next we prove Lemma \ref{integral:lemma}.
\begin{proof}[Proof of Lemma \ref{integral:lemma}]
The proof is similar to \cite[Lemma 4.3]{LS}.
Recall \eqref{H-integral} which we relabel as
\bea\label{H-relabel}
\mathcal{H}(x)=\int_{\mathbb{R}}
\omega\left(\xi\right)
\mathfrak{I}^*\left(M\xi,n_1,q\right)
\overline{\mathfrak{I}^*\left(M\xi,n_2,q\right)}
\, e\left(-x\xi\right)\mathrm{d}\xi,
\eea
where by \eqref{I*},
\bea\label{I*-2}
\mathfrak{I}^*(M\xi,n,q)
=\frac{1}{\sqrt{t}}G_{\natural}(y_*)
e\left(By_0+\frac{y_0^2}{2c\beta^2}\frac{B^2}{t}+B\sum_{j=2}^{K_2}g_{c,\beta,j}\left(y_0\right)
\left(\frac{B}{t}\right)^j\right) + O_{A}(t^{-A}).
\eea
Here $y_0, y_*$ are as in \eqref{stationary point}, $G_{\natural}(x)$ is some inert function
supported on $x\asymp 1$, $B=2q^{-1}(nX)^{1/2}$ is defined in \eqref{BD} and $g_{c,\beta, j}(x)$ some polynomial function
depending only on $c,\beta,j$.
Trivially, one has
\bna
\mathcal{H}(x)\ll t^{-1}.
\ena
This proves the first statement of Lemma \ref{integral:lemma}.

Plugging \eqref{I*-2} into \eqref{H-relabel}, we obtain
\bna
&&\mathcal{H}(x)=\frac{1}{t}
\int_{\mathbb{R}}
\omega\left(\xi\right)G_{\natural}(y_*)\overline{G_{\natural}(y_*')}
e\left(-x\xi+(B-B')\widetilde{y}_0\xi^{-1/(2\beta)}+
(B^2-B'^2)\frac{\widetilde{y}_0^2}{2c\beta^2t}\xi^{-1/\beta}\right)\\
&&\qquad\qquad\times e\left(\sum_{j=2}^{K_2}g_{c,\beta,j}(\widetilde{y}_0\xi^{-1/(2\beta)})
\bigg(B\bigg(\frac{B}{t}\bigg)^j
-B'\bigg(\frac{B'}{t}\bigg)^j\bigg)\right)
\mathrm{d}\xi + O_{A}( t^{-A}),
\ena
where $\widetilde{y}_0=y_0\xi^{1/(2\beta)}=(ct/\widetilde{D})^{1/\beta}\asymp 1$ with
$\widetilde{D}=D\xi^{-1/2}=2q^{-1}(MX)^{1/2}$ is defined in \eqref{BD},
$y_0, y_*$ are as in \eqref{stationary point},
and $B$ is defined in \eqref{BD} and $B'$ is defined in the same way but with
$n_1$ replaced by $n_2$. Note that the first derivative of
the phase function in the above integral equals
\bea\label{1st phase function}
&&-x-\frac{1}{2\beta}(B-B')\widetilde{y}_0\xi^{-1/(2\beta)-1}-\frac{1}{\beta}
(B^2-B'^2)\frac{\widetilde{y}_0^2}{2c\beta^2t}\xi^{-1/\beta-1}\nonumber\\&&
-\frac{1}{2\beta}\widetilde{y}_0\xi^{-1/(2\beta)-1}\sum_{j=2}^{K_2}g'_{c,\beta,j}(\widetilde{y}_0\xi^{-1/(2\beta)})
\bigg(B\bigg(\frac{B}{t}\bigg)^j
-B'\bigg(\frac{B'}{t}\bigg)^j\bigg)
\eea
which is $\gg |x|\gg X^{\varepsilon}$
if $|x|\gg X^{\varepsilon}\sqrt{XN_1}/C\asymp X^{1+\varepsilon}\Xi/(CQ)$
since $B, B'\asymp \sqrt{XN_1}/C$ and $N_1\asymp  X\Xi^2/Q^2$ in \eqref{M-N1-range}.
Then repeated integration by parts shows that
the contribution from $x\gg  X^{1+\varepsilon}\Xi/(CQ)$ is negligible.
Thus the second statement of Lemma \ref{integral:lemma} is clear.

Moreover, if $-1/(2\beta)-1\neq 0$, i.e., $\beta\neq -1/2$ or equivalently,
$\varphi(x)\neq cx^{3/4}$, the second term in \eqref{1st phase function}
is of size
\bna
|B-B'|=\frac{2N^{1/2}}{q}|n_1^{1/2}-n_2^{1/2}|
\asymp \frac{X^{1/2}}{CN_1^{1/2}}|n_1-n_2|\asymp \frac{Q}{C\Xi}|n_1-n_2|
\ena
since $N_1\asymp  X\Xi^2/Q^2$. Thus repeated integration by parts shows that
$\mathcal{H}(x)$ is negligibly small unless $|x|\asymp \frac{Q}{C\Xi}|n_1-n_2|$.
Now by applying the second derivative test
in Lemma \ref{lem: 2st derivative test, dim 1}, we infer that for $x\neq 0$ and
$\varphi(x)\neq cx^{3/4}$,
\bna
\mathcal{H}(x)\ll t^{-1} |x|^{-1/2}.
\ena
This proves (3).

Finally, for $x=0$, using the identity $a^{j+1}-b^{j+1}=(a-b)(a^j+a^{j-1}b+\cdots+ab^{j-1}+b^j)$
and \eqref{B upper bound}, one sees that, for $j\geq 1$,
\bna
B\bigg(\frac{B}{t}\bigg)^j
-B'\bigg(\frac{B'}{t}\bigg)^j&=&(B-B')\left(\bigg(\frac{B}{t}\bigg)^j+
\bigg(\frac{B}{t}\bigg)^{j-1}\frac{B'}{t}+\cdots+\frac{B}{t}\bigg(\frac{B'}{t}\bigg)^{j-1}
+\bigg(\frac{B'}{t}\bigg)^j
\right)\\
&\ll& |B-B'|X^{-\varepsilon}.
\ena
Thus the first derivative
of the phase function in \eqref{1st phase function} is
\bna
\gg |B-B'|\asymp \frac{Q}{C\Xi}|n_1-n_2|.
\ena
By repeated integration by parts, $\mathcal{H}(0)$
is negligible small unless $|n_1-n_2|\ll C\Xi N^{\varepsilon}/Q$.
Since $\Xi\ll N^{\varepsilon}$ and $C\ll Q$, we have that
$\mathcal{H}(0)$
is negligibly small unless $|n_1-n_2|\ll N^{\varepsilon}$.
This completes the proof of Lemma \ref{integral:lemma}.
\end{proof}

\end{document}